%% file: ex_article.tex
\documentclass[a4paper,fleqn]{els-cas-templates/cas-sc}

\usepackage[numbers,sort&compress]{natbib}

\input{ex_shared}

\input{commands}

\usepackage{caption}
\usepackage{stmaryrd}
\usepackage{bm}
\usepackage{tikz}
\renewcommand{\figurename}{Figure}

\begin{document}
\let\WriteBookmarks\relax
\def\floatpagepagefraction{1}
\def\textpagefraction{.001}
\ExplSyntaxOn
\bool_gset_true:N \g_stm_nologo_bool
\ExplSyntaxOff

\shorttitle{CR FEM for Time-dependent transport equations}
\shortauthors{S. Mao, M. Zhang}

\title[mode=title]{A Second-Order Maximum-Principle-Preserving Crouzeix-Raviart Finite Element Method for Time-dependent Transport Equation}

\author[1]{Shipeng Mao}
\fnmark[1]
\ead{maosp@lsec.cc.ac.cn}

\author[1]{Mingyang Zhang}
\fnmark[1]
\ead{zhangmingyang@lsec.cc.ac.cn}

\affiliation[1]{organization={State Key Laboratory of Mathematical Sciences, Academy of Mathematics and Systems Science, Chinese Academy of Sciences; School of Mathematical Sciences, University of Chinese Academy of Sciences},
                addressline={Zhongguancun East Road 55},
                city={Beijing},
                postcode={100190},
                country={China}}

\begin{abstract}
	In this paper, we construct an explicit, second-order, and maximum-principle-preserving Crouzeix-Raviart (CR) finite element method for two-dimensional time-dependent transport equation. The key observation is that the mass matrix  of the CR element is with diagonal structure, which allows us to avoid the need to solve a large linear system for each time step and help to construct a low-order scheme that preserves the maximum principle in a simple way. We first introduce low-order schemes based on minimum and bilinear viscosities, and then recover second-order accuracy by means of greedy and flux-corrected transport viscosities. For inflow boundary conditions, we further design a modified FCT limiter. In addition, we propose a simple reconstruction based on Wachspress coordinates to obtain a continuous piecewise linear approximation on the $\frac{h}2$-mesh that satisfies the maximum principle on the whole domain. Under divergence-free velocity fields, the proposed schemes are conservative. Numerical experiments on both smooth and discontinuous test cases, with both solenoidal and non-solenoidal velocity fields, confirm the accuracy and robustness of the proposed schemes.
\end{abstract}



\maketitle

\section{Introduction}
The time-dependent transport equation encodes directional advection and hyperbolic wave phenomena that arise in acoustics, aerodynamics, and other application areas where information propagates along characteristics.
In applications, low-regularity or incompatible initial-boundary data can propagate along characteristics, producing steep profiles and weak singular features in the solution. Purely centered or unstabilized discretizations are then prone to spurious oscillations and to pronounced numerical dispersion or dissipation \cite{Johnson1984}.
Godunov's seminal construction \cite{MR119433} advances the solution through local Riemann problems with upwind piecewise-constant Godunov fluxes, a pattern central to shock capturing and to exact/approximate Riemann-solver-based discretizations of compressible transport.
By way of orientation rather than an exhaustive survey, finite volume formulations may incorporate flux and slope limiters \cite{MR1388492}, continuous finite elements often use streamline upwind/Petrov--Galerkin (SUPG) and related Petrov--Galerkin modifications, crosswind dissipation and discontinuity capturing, or spurious oscillation at layers diminishing (SOLD) strategies \cite{MR679322,MR571681,MR1256324,John2007,John2008}, discontinuous Galerkin (DG) methods frequently stabilize advection through numerical fluxes such as upwind or monotone choices \cite{Meng2016,Xu2019}, and weak Galerkin schemes illustrate weak-gradient advection representations together with edge-based stabilizers \cite{Zhai2018, ALTaweel2021}.
Beyond stability, an important property of transport schemes is to respect maximum principles (or, more generally, invariant domains) so that undershoots and overshoots cannot produce unphysical states.
Godunov's classical theorem \cite{MR119433} shows that linear monotone schemes are at most first-order accurate, so approximations that suppress spurious extrema while seeking higher resolution necessarily rely on nonlinear stabilization mechanisms.
Recent designs therefore combine flux-corrected transport (FCT) corrections with residual-based dissipation to algebraically enforce discrete maximum principles \cite{Hajduk2020,MR4097213}.

The goal of this paper is to  develop an explicit second-order Crouzeix--Raviart (CR) finite element method for two-dimensional time-dependent transport equation that preserves a discrete maximum principle.
Introduced in 1973, the CR element is a nonconforming piecewise-linear space on simplices  with degrees of freedom tied to edge midpoints and enforced 
continuity in an integral (average) sense along mesh faces \cite{MR343661}.
It has attracted considerable attention within the finite element community, notably in inf--sup stable lowest-order mixed discretizations of incompressible flow \cite{MR1721268,MR1284845, MR3194820, MR2802555, MR3010181}, in locking-robust elasticity \cite{MR1182731,MR1140646, MR1202659} and in compressible fluid models \cite{MR2501053, MR3117509, MR3563278, MR4344597}.
For further background on CR elements and related extensions, see Brenner's survey article \cite{MR3312124}.
Motivations behind this spatial discretization for the transport equation stem from the following facts.
In explicit finite element methods for time-dependent transport problems, each time step typically requires a mass-matrix linear solve. For continuous $P_1$ finite elements, mass lumping is therefore widely used to obtain a diagonal mass matrix, reducing the cost of these solves and facilitating maximum-principle-preserving formulations \cite{MR3171280,MR4050715}. However, lumping introduces numerical dispersion \cite{MR3738312}, which can be particularly harmful for transport-like equations with nonsmooth initial data.
We found that the Crouzeix--Raviart mass matrix is naturally diagonal on any two-dimensional simplicial mesh, which eliminates the need for mass lumping and simplifies the design of maximum-principle-preserving schemes.
The semidiscrete advection operator also uses a sparser stencil than continuous $P_1$ at comparable order, weakening neighbor coupling and reducing halo traffic in parallel implementations.
Compared with piecewise-linear DG on the same triangulation, CR uses far fewer global unknowns, since DG stores three independent $P_1$ coefficients on each triangle whereas CR stores only one value per edge.
The smaller global dimension shortens the discrete state vector, saves memory for the solution and auxiliary data structures, and lowers the arithmetic work per explicit stage on a fixed mesh.
Moreover, the CR finite element space possesses an intermediate level of regularity, with smoothness lying between $L^2(\Omega)$ and $H^1(\Omega)$, and can approximate functions that do not belong to $H^1(\Omega)$.
With these features, the CR element may provide a desirable balance between continuous and discontinuous finite element methods for the transport equation, making its study worthwhile.

In this paper, we firstly introduce two types of viscosity to construct low-order schemes that preserve a local maximum principle at the degree-of-freedom (DOF) level, ensuring each DOF remains bounded by the extrema of its neighboring DOFs from the previous time level. The two low-order viscosities are as follows:
\begin{enumerate}
	\item Minimum viscosity: minimizes the entries of the viscosity matrix while maintaining essential numerical properties, including symmetry and conservation.
	\item Bilinear viscosity: formulated as a bilinear form on the CR basis functions, sharing conceptual similarities with the tensor-valued viscosity proposed in \cite{MR3171280}.
\end{enumerate}
To recover higher-order accuracy, we then incorporate two viscosity reduction mechanisms following \cite{MR3738312}:
\begin{enumerate}
	\item Greedy viscosity: applies an element-wise scaling factor between 0 and 1 to the low-order viscosity matrix.
	\item Flux corrected transport (FCT) viscosity: derived from the flux correction technique \cite{MR1486270,MR534786}, this limiter uses the maximum-preserving low-order solution to correct the high-order approximation.
\end{enumerate}
We prove that both greedy and FCT viscosities preserve the local maximum principle at the DOFs. By appropriately combining low- and high-order viscosities, we obtain schemes that are second-order accurate for smooth solutions. Under a divergence-free velocity field, all the methods are conservative. To handle inflow boundary conditions, we further construct a modified FCT limiter, which preserves the maximum principle and retains second-order accuracy in our numerical tests.
The resulting CR solution ensures that all degrees of freedom stay within the bounds of the initial and inflow data. To obtain a bound-preserving approximation on the whole domain, we employ Wachspress coordinates \cite{MR426460,MR3218572} to reconstruct a Lagrange $P_1$ function on the refined mesh obtained by subdividing each element at edge midpoints into four sub-triangles. This reconstruction satisfies the maximum principle throughout the domain, retains second-order accuracy, and has its $L^2$ norm controlled by that of the original CR solution. Numerical results in this paper confirm the accuracy and robustness of the proposed method and also indicate that, in challenging transport regimes, the CR discretization can exhibit smaller dispersion errors than the continuous $P_1$ method.

The paper is organized as follows. \Cref{sec:pre} introduces the model problem and notation. In \Cref{sec:crfem}, we present the CR method, including its weak and matrix formulations and the temporal discretization. \Cref{sec:maximumprinciple} develops the low- and high-order viscosities and introduces a limiting technique for inflow boundary conditions. \Cref{sec:postprocess} presents the CR reconstruction and verifies its second-order accuracy. \Cref{sec:numericaltests} provides numerical validation, and \Cref{sec:conclusions} concludes the paper.

\section{Preliminaries}
\label{sec:pre}
\subsection{Formulation of the problem}
Let $\Omega$ be an open polygonal domain in $\mathbb{R}^2$.  We consider the following time-dependent transport equation
\begin{equation}
	\begin{cases}
		\p_t u+\bm{\beta}\cdot\nabla u=0,\quad (\mathbf{x},t)\in\Omega\times \R_+,\\
		u(\mathbf{x},0)=u_0(\mathbf{x}),\quad \mathbf{x}\in\Omega,\\
		u|_{\p\Omega^{-}} = u_{\text{in}}(\mathbf{x},t),\quad (\mathbf{x},t)\in\p\Omega^{-}\times\R_+,
	\end{cases}
	\label{model}
\end{equation}
where $\bm{\beta}\in \mathbf{L}^{\infty}(\Omega)$ is the velocity field, and the inflow boundary is defined by $\partial \Omega^- := \{\mathbf{x} \in \partial \Omega : \bm{\beta} \cdot \mathbf{n} < 0\}$ with $\mathbf{n}(\mathbf{x})$ the outward unit normal. The initial data $u_0 \in L^\infty(\Omega)$ and inflow data $u_{\text{in}} \in L^{\infty}(\partial \Omega^-)$. 
The solution $u$ is bounded by the extrema of the initial and inflow data, i.e., $\min(\inf_{\Omega}u_0,\inf_{\p\Omega^-}u_{\text{in}})\le u(\mathbf{x},t)\le \max(\sup_{\Omega}u_0,\sup_{\p\Omega^-}u_{\text{in}})$ for any $\mathbf{x}\in\Omega$ and $t>0$. Moreover, if 
$\bm{\beta}$ is also divergence-free, then the flow satisfies the conservation law, i.e., $\int_{\Omega}u(\mathbf{x},t)\mathrm{d}\mathbf{x}$ is constant.

\subsection{Notations}
Let $\{\mathcal{K}_h\}_{h>0}$ denote a family of triangular meshes of the polygonal domain $\Omega$, assumed to be shape-regular in the sense of Ciarlet and without hanging nodes. Let $\hat{K}$ denote the reference triangle, taken to be equilateral with unit side length. The affine mapping between $\hat{K}$ and an arbitrary trangle $K\in \mathcal{K}_h$ is denoted by $\Phi_K:\hat{K}\to K$ which is obviously affine. Let $\mathcal{F}_h$ be the set of all edges of the mesh $\mathcal{K}_h$, and let $\mathcal{F}_h^{\mathrm{int}}\subset \mathcal{F}_h$ denote the interior edges. The set of edge midpoints is denoted by $\{\mathbf{x}_i\}_{i=1}^N$, where $N$ is the total number of edges. For each triangle $K \in \mathcal{K}_h$, let $h_K$ be the diameter of its circumscribed circle and define  $h = \max_{K\in\mathcal{K}_h} h_K$.

The Crouzeix-Raviart finite element space is defined as
\begin{equation}
	\mathcal{V}_h =\{v_h\in P(\mathcal{K}_h):\int_F\llbracket v_h\rrbracket\mathrm{d}s=0,\forall\ F\in\mathcal{F}_h^{\mathrm{int}}\},
\end{equation}
where $P(\mathcal{K}_h): = \{v_h\in L^1(\Omega):v_h|_K\circ \Phi_K\in \hat{P}_1,\forall\ K\in\mathcal{K}_h\}$, $\hat{P}_1$ is the space of linear functions on the reference element $\hat{K}$, and $\llbracket v_h\rrbracket:= v_h|_F^+-v_h|_F^-$ is the jump of $v_h$ across an interior edge with
$v_h|_F^-:=\lim_{\epsilon\rightarrow 0^+}v_h(\mathbf{x}-\epsilon\mathbf{n}_F)$, $v_h|_F^+:=\lim_{\epsilon\rightarrow 0^+}v_h(\mathbf{x}+\epsilon\mathbf{n}_F)$.

Let $\{\varphi_1,\varphi_2,\cdots,\varphi_N\}$ denote the CR shape functions, which are piecewise linear and satisfy the nodal property at edge midpoints: $
	\varphi_i(\mathbf{x}_j) =\delta_{ij}$ for $i,j=1,2,\cdots,N$.
Any  $v_h \in \mathcal{V}_h$ can then be expressed as $v_h = \sum\limits_{i=1}^N v_i\varphi_i$.
Denote the support region of $\varphi_i$ as $S_i$  and  and the intersection of two supports by $S_{ij}:=S_i\cap S_j$.  For a union of cells $E \subset \mathcal{K}_h$, let $\mathcal{I}(E) := \{ j : |S_j \cap E| \neq 0 \}$ be the set of indices of shape functions whose support intersects $E$. The area of $E$ is denoted by $|E|$, and the cardinality of $\mathcal{I}(E)$ by $\mathrm{card}(\mathcal{I}(E))$.
  
\section{Crouzeix-Raviart Finite Element Method}
\label{sec:crfem}
\subsection{Semi-discretization}
Following a weak formulation similar to the discontinuous Galerkin method \cite{MR1619652,MR3311872}, with boundary conditions imposed weakly, the semi-discretization of (\ref{model}) reads: find $u_h \in \mathcal{V}_h$ such that
\begin{equation}
	\int_{\Omega}\p_t u_hv_h\mathrm{d}\mathbf{x}+a_h(u_h,v_h)= b_h(u_h,v_h)+l_h(u_h,v_h)\quad \forall\ v_h\in \mathcal{V}_h,\label{weakform} 
\end{equation}
where 
\begin{subequations}
\begin{align}
&a_h(u_h,v_h):=	\sum_{K\in\mathcal{K}_h}\int_{K}(\bm{\beta}u_h)\cdot\nabla v_h\mathrm{d}\mathbf{x},\label{ah}\\
&b_h(u_h,v_h):=\sum_{K\in\mathcal{K}_h}\int_{\p K-\p\Omega}\bm{\beta}\cdot\mathbf{n}u_hv_h\,\mathrm{d}s -\sum_{F\in\mathcal{F}_h^{\mathrm{int}}}\int_F\mathrm{Up}(u_h,\bm{\beta})(-\llbracket v_h\rrbracket)\, \mathrm{d} s,\label{bh}\\
&l_h(u_h,v_h):=\int_{\p \Omega^-}\bm{\beta}\cdot \mathbf{n}(u_h-u|_{\text{in}})v_h\,\mathrm{d}s.\label{lh}
\end{align}
\end{subequations}
The upwind flux is defined by
\begin{equation}\mathrm{Up}(u_h,\bm{\beta})|_F:=u_{h,F}^-[\bm{\beta}\cdot\mathbf{n}_F]^++u_{h,F}^+[\bm{\beta}\cdot\mathbf{n}_F]^-,\end{equation}
with 
$[\bm{\beta}\cdot\mathbf{n}_F]^+:=\max(0,\bm{\beta}\cdot\mathbf{n}_F)$, $[\bm{\beta}\cdot\mathbf{n}_F]^-:=\min(0,\bm{\beta}\cdot\mathbf{n}_F)$,
and $\mathbf{n}_F$ denotes the fixed unit normal vector on the edge $F$. 

Here, $b_h(u_h,v_h)$ represents the consistency error of the divergence term, which vanishes for $H^1$-conforming elements, and $l_h(u_h,v_h)$ accounts for the weakly imposed inflow boundary condition. Without these two terms, the scheme is observed to suffer from a loss of both accuracy and stability.

\subsection{Time discretization}
\label{sec:timediscretization}
 Let $u_h^n=\sum_{i=1}^NU_i^n\varphi_i$ approximate $u(\mathbf{x},t^n)$, with $U^n = (U_1,U_2,\cdots,U_N)^T$ and time step $\Delta t ^n=t^{n+1}-t^n$. The forward Euler method leads to
 \begin{equation}
    M(U^{n+1}-U^n) =\Delta t^n F^n,\label{forwardeuler}
 \end{equation}
where the mass matrix $M$ has entries $m_{ij} = \sum_{K\in\mathcal{K}_h} \int_K \varphi_i \varphi_j \, d\mathbf{x}$, and
 \begin{equation}
    F^n_i = -a_h(u_h^n,\varphi_i)+b_h(u_h^n,\varphi_i)+l_h(\varphi_i), \quad i=1,2,\cdots,N.\label{F}
 \end{equation}
 Note that the mass matrix $M$ is diagonal, i.e., $m_{ij} = \delta_{ij}m_i$, where $m_i := \frac13|S_i|$. The scheme therefore simplifies to
\begin{equation}
    U_i^{n+1} = U_i^n+\frac{\Delta t^n}{m_i}F_i^n,\quad i=1,2,\cdots,N.\label{forward}
\end{equation}
For high-order time discretization, we resort to strong stability preserving Runge-Kutta method (SSP RK) \cite{MR1854647}.  All numerical tests in this paper are based on the following SSP RK(3,3) method:
\begin{equation}
    \begin{aligned}
        &y^{(1)}=y^n+\Delta t^nf(y^{(n)}),\\
        &y^{(2)}=\frac34y^n+\frac14(y^{(1)}+\Delta tf(t^n+\Delta t^n,y^{(1)})),\\
        &y^{n+1}=\frac13y^n+\frac23(y^{(2)}+\Delta tf(t^n+\frac12\Delta t^n,y^{(2)})).
    \end{aligned}
\end{equation}
Each stage of SSPRK(3,3) can be expressed as a convex combination of forward Euler substeps; hence, it preserves the maximum principle whenever the forward Euler scheme does. The main objective of this paper is to develop maximum-principle-preserving schemes based on (\ref{weakform}).

\section{Maximum-principle-preserving scheme}
\label{sec:maximumprinciple}
In this section, we first develop a low-order maximum-principle-preserving scheme by adding appropriate artificial viscosity to the forward Euler method. We then introduce two techniques for reducing the viscosity to achieve high-order accuracy while maintaining the maximum principle at the degrees of freedom.

Assume that  there's no inflow boundary, then 
$F_i$ in (\ref{F}) can be rewritten as
\begin{equation}F_i^n = \sum_{j\in\mathcal{I}(S_i)}(-a_{ij}^n+b_{ij}^n)U_j^n,\end{equation}
where $a_{ij}^n:=	a_h(\varphi_j,\varphi_i)$, $b_{ij}^n:=b_h(\varphi_j,\varphi_i)$.
Then (\ref{forward}) reads
\begin{equation}
	U^{n+1} = U^n-\Delta t^n M^{-1}S^nU^n=(I-\Delta t^n M^{-1}S^n)U^n,\label{forwardA}
\end{equation}
where $S^n:= A^n-B^n$ with $A^n = (a_{ij}^n)$ and $B^n = (b_{ij}^n)$.

\begin{theorem}[Mass conservation]\label{thm:conservation}
	If the boundary is impermeable, i.e., $\bm{\beta}\cdot \mathbf{n}=0$ on $\partial\Omega$, if $\bm{\beta}$ is divergence-free, and if the normal component of $\bm{\beta}$ is continuous across each interior edge, then the scheme \eqref{forwardA} is conservative, i.e.,
	\[
	\int_{\Omega}u_h^n\,\mathrm{d}\mathbf{x}=\int_{\Omega}u_h^0\,\mathrm{d}\mathbf{x}
	\qquad\text{for all } n\ge 0.
	\]
\end{theorem}
\begin{proof}
	From (\ref{forwardA}), we have
	\[\int_{\Omega}u_h^{n+1}\,\mathrm{d}\mathbf{x}-\int_{\Omega}u_h^{n}\,\mathrm{d}\mathbf{x}=\sum_{i=1}^N m_i(U_i^{n+1}-U_i^n)
		=-\Delta t^n\sum_{i,j}(-a_{ij}^n+b_{ij}^n)U_j^n=0.\]
	The last equality follows since $\sum\limits_i(-a_{ij}^n+b_{ij}^n)=0$ for all $j$.
\end{proof}

\subsection{Low-order viscosity}
\label{sec:viscosity}
To construct a low-order method, we introduce a viscosity matrix 
$V^n$ added to $S^n$ in (\ref{forwardA}) such that  $W^n:=I-\Delta t^n M^{-1}(S^n-V^n)$ is nonnegative with row sums equal to 1. In this case,  \( U^{n+1} \)  becomes a convex combination of the previous values \( U^n \), which ensures that the maximum principle with respect to DOFs is preserved.
Following \cite{MR3738312}, $V^n$ is required to satisfy the following properties:
\begin{enumerate}
	\item Symmetry: \( v_{ij}^n = v_{ji}^n \) for all \( i, j = 1, 2,\cdots, N \).
	\item Conservation: \( \sum_{j=1}^N v_{ij}^n = 0 \) for all \( i = 1,  2,\cdots, N \).
	\item Nonnegativity: \( v_{ij}^n \ge s_{ij}^n \) for all \( i \neq j \).
\end{enumerate}
The viscosity matrix with the minimal entries that satisfies the above conditions can be defined as follows
\begin{equation}
	v_{ij}^n :=\begin{cases} \max(0,s_{ij}^n,s_{ji}^n),&\text{if }  j\ne i, \\
	-\sum\limits_{j\ne i}v_{ij}^n,& \text{if }i=j . \end{cases}\label{minvisc}
\end{equation}
We refer to this as the minimum viscosity.
A more diffusive but structurally convenient variant, termed the bilinear viscosity, is inspired by the tensor-valued formulation in \cite{MR3171280}:
\begin{equation}
v_{ij}^n :=\begin{cases} \max(0,\max\limits_{l,k\in\mathcal{I}(S_{ij}),l\ne k}s_{lk}^n)&\text{if }i\ne j\\-2\sum\limits_{K\subset S_i}\max(0,\max\limits_{l,k\in\mathcal{I}(S_{ij}),l\ne k}s_{lk}^n)& \text{if }i=j,\end{cases}\label{bilinearvisc}
\end{equation}
which admits a bilinear representation. To see how it can be rewritten as a bilinear form, we first define  $\nu_K^n$ the viscosity on each element as
\begin{equation}
\nu_K^n := \max_{i\ne j \in \mathcal{I}(K)}\frac{\max(0,s_{ij}^n)}{-r_K(\varphi_i,\varphi_j)},
\end{equation}
where $r_K(\cdot,\cdot)$ satisfies the same properties as $b_K(\cdot,\cdot)$ in \cite{MR3249370}.
Then 
\begin{equation}v_{ij}^n = -\sum_{K\subset S_{ij}}\nu_K^n r_K(\varphi_i,\varphi_j), \quad j\in\mathcal{I}(S_i),\quad i=1,2,\cdots,N.\end{equation}
and, for any $u_h,v_h\in \mathcal{V}_h$, we define
\begin{equation}
	r_h(u_h,v_h) :=\sum_{K\in\mathcal{K}_h}\sum_{i,j\in\mathcal{I}(K)}u_jv_i\nu_K^n r_K(\varphi_j,\varphi_i).
\end{equation}
Thus, adding the bilinear viscosity matrix into (\ref{forwardA})  is equivalent to adding the bilinear form $b_h(u_h,v_h)$  into the left hand side of the weak formulation (\ref{weakform}).

\begin{remark}
For the CR element, each basis function $\varphi_i$ has at most five neighboring degrees of freedom (DOFs), located on no more than two adjacent cells.  This localized connectivity of the CR element provides significant advantages in parallel computing and leads to a sparser matrix, reducing both communication and computational costs.
\end{remark}

\begin{theorem}[Local discrete maximum principle with respect to DOFs]\label{thm:localmaxprinciple}
Let $u_{\min}:=\inf_{\mathbf{x}\in\Omega}u_0(\mathbf{x})$, $u_{\max}:=\sup_{\mathbf{x}\in\Omega}u_0(\mathbf{x})$. Assume that $u_{\min}\le U_i^0\le u_{\max}$, and $\norm{\bm{\beta}}_{L^{\infty}(\Omega\times\R^+)}< \infty$. Then, with minimum viscosity matrix (\ref{minvisc}) applied,  the scheme satisfies a local discrete maximum principle with respect to DOFs, i.e. $u_{\min}\le \min_{j\in\mathcal{I}(S_i)}U_j^n\le \max_{j\in\mathcal{I}(S_i)}U_j^n\le u_{\max}$ for all time step $n\ge 0$, provided that $\Delta t^n$ is sufficiently small. More precisely, the condition
\begin{equation}
	\Delta t^n\le \min_{i=1,2,\cdots,N}\frac{m_i}{s_{ii}^n-v_{ii}^n},\label{CFL}
\end{equation}
is sufficient.
\end{theorem}
\begin{proof}
	We proceed by induction.
For $n=0$, the claim holds by the initial condition.
Assume it holds for step $n$. From the definition of the minimum viscosity (\ref{minvisc}), we have for $i\ne j$:
	\[w_{ij}^n=-\frac{\Delta t^n}{m_i}(s_{ij}^n-v_{ij}^n)\ge 0,\]
	and for $i=j$, we have
	\[w_{ii}^n=1-\frac{\Delta t^n}{m_i}(s_{ii}^n-v_{ii}^n).\]
	Hence if $\Delta t^n\le \frac{m_i}{s_{ii}^n-v_{ii}^n}$, then $w_{ii}^n\ge 0$. Moreover,
	\[\sum_{j\in\mathcal{I}(S_i)}w_{ij}=1-\frac{\Delta t^n}{m_i}(\sum_{j\in\mathcal{I}(S_i)}s_{ij}^n-\sum_{j\in\mathcal{I}(S_i)}v_{ij}^n)=1,\]
	so each $U_i^{n+1}$ is a convex combination of ${U_j^n}_{j\in\mathcal{I}(S_i)}$.
The induction hypothesis then implies the result for $n+1$.
\end{proof}

\begin{remark}
A similar argument shows that the local discrete maximum principle also holds when the bilinear viscosity (\ref{bilinearvisc}) is used instead of the minimum viscosity.
\end{remark}

\begin{theorem}[Mass conservation]
	Assume that the conditions in Theorem \ref{thm:conservation} hold, that the viscosity matrix satisfies $v_{ij}^n=v_{ji}^n$, and that $\sum_{j\in\mathcal{I}(S_i)}v_{ij}^n=0$. Then the scheme
	\begin{equation}
		U^{n+1}=(I-\Delta t^n M^{-1}(S^n-V^n))U^n\label{forward_mass}
	\end{equation}
	is conservative, i.e.,
	\[
	\int_{\Omega}u_h^n\,\mathrm{d}\mathbf{x}=\int_{\Omega}u_h^0\,\mathrm{d}\mathbf{x}
	\qquad\text{for all } n\ge 0.
	\]
\end{theorem}

\begin{proof}
	From (\ref{forward_mass}), we have
\[\int_{\Omega}u_h^{n+1}\,\mathrm{d}\mathbf{x}-\int_{\Omega}u_h^{n}\,\mathrm{d}\mathbf{x}
	=\sum_{i=1}^N m_i(U_i^{n+1}-U_i^n)
	=\Delta t^n\sum_{i,j}v_{ij}^n(U_j^n-U_i^n)=0.\]
	The last equality follows from the symmetry of $V^n$ and the property $\sum_{j\in\mathcal{I}(S_i)}v_{ij}^n=0$.
\end{proof}

\subsection{Higher-order viscosity}
In this section, we introduce two techniques to enhance the low-order scheme introduced in Section \ref{sec:viscosity}. The goal is to reduce the solution error on a given mesh and achieve higher accuracy while preserving the discrete maximum principle at the level of DOFs.

\label{sec:highorderviscosity}
\subsubsection{Preliminary lemma}
We begin by defining the local extrema of $u_h^n$ at the nodal points $\{\mathbf{x}_j\}_{j=1}^N$ within the support of each basis function $\varphi_i$:
\begin{equation}
  U^{M,n}_i = \max_{j\in\mathcal{I}(S_i)} U_j^n,\quad
  U^{m,n}_i = \min_{j\in\mathcal{I}(S_i)} U_j^n.
\end{equation}
Then, we introduce the normalized parameter
\begin{equation}
	\theta_i^n := \begin{cases} 
\frac{U_i^n - U_i^{m,n}}{U_i^{M,n} - U_i^{m,n}} & \text{if } U_i^{M,n} - U_i^{m,n} \neq 0, \\
\frac{1}{2} & \text{otherwise.}
\end{cases}
\end{equation}
Let $V^{L,n}=(v_{ij}^{L,n})$ denote the low-order viscosity matrix and define the high-order viscosity in the form
\begin{equation}
	V^{H,n} = \Psi^{n}*V^{L,n},
\end{equation}
where $*$ denotes elementwise multiplication, and $\Psi^n=(\psi_{ij}^n)$ is a scaling matrix determined by a vector $\psi^n=(\psi_i^n)$ with entries in $[0,1]$:
\begin{equation}
	\psi_{ij}^n = \max(\psi_i^n,\psi_j^n),\quad  i\ne j, \ i,j=1,2,\cdots,N.\label{psiij1}
\end{equation} 
For the diagonal entries, we enforce consistency through
\begin{equation}
	\psi_{ii}^n = \begin{cases}-\frac{\sum_{j\ne i}\psi^n_{ij} v_{ij}^{L,n}}{v_{ii}^{L,n}}&\text{if }v_{ii}^{L,n}=0,\\
	0&\text{otherwise},\end{cases}\quad i=1,2,\cdots,N.\label{psiij3}
\end{equation}
Define $\mathcal{I}(S_i^+) = \{j \in \mathcal{I}(S_i) \mid \mathsf{U}_i^n < \mathsf{U}_j^n\},\ \mathcal{I}(S_i^-) = \{j \in \mathcal{I}(S_i) \mid \mathsf{U}_i^n > \mathsf{U}_j^n\}$ and introduce the parameters
\begin{equation}
		\gamma_i^n :=\frac{\Delta t^n}{m_i}(s_{ii}^n-v_{ii}^{L,n}),
\quad \gamma_i^{\pm,n} := \frac{\Delta t^n}{m_i} \sum_{j \in \mathcal{I}(S_i^{\pm})}| v_{ij}^{L,n}|.
\end{equation}
With these definitions and notations in place, we can now state the following lemma.
\begin{lemma}\label{psi}
	{Assume that}
$k\gamma_i^n < 1,k\in\mathbb{N}^+$  and 
$U_{i}^{M,n} - U_{i}^{m,n} \neq 0, \ n \geq 0, \ i \in \{1,2,\cdots,N \}$; then
\begin{align}
	& U_{i}^{H,n+1} \leq U_{i}^{M,n} - (U_{i}^{M,n} - U_{i}^{n}) \left((1 - \theta_{i}^{n})(1 - k\gamma_{i}^{n}) - \theta_{i}^{n}(1 - \psi_{i}^{n}) \gamma_{i}^{-,n} \right), \\
& U_{i}^{H,n+1} \geq U_{i}^{m,n} + (U_{i}^{M,n} - U_{i}^{n}) \left( \theta_{i}^{n} (1 -k \gamma_{i}^{n}) - (1 - \theta_{i}^{n})(1 - \psi_{i}^{n})\gamma_{i}^{+, n} \right).
\end{align}
\end{lemma}
\begin{proof}
	The low-order update can be expressed as
	\begin{equation}
	\begin{aligned}
	U_{i}^{L,n+1}& = U_i^n - \frac{\Delta t^n}{m_i} \sum_{j \in \mathcal{I}(S_i)} (s_{ij}^n -v_{ij}^{L,n}) U_j^n\\
	& = U_i^{n}(1-k\gamma_i^n)+(k-1)\gamma_i^n U_i^n+\frac{\Delta t^n}{m_i}\sum_{j\in\mathcal{I}(S_i),j\ne i}|s_{ij}^n-v_{ij}^{L,n}|U_j^n.
	\end{aligned}\label{low-order}
	\end{equation}
	From the relations
	$s_{ij}^n-v_{ij}^{L,n} \le 0\ (\forall\ j\ne i),\quad
	s_{ii}^n=-\sum\limits_{j\ne i}s_{ij}^n,\quad
	v_{ii}^{L,n}=-\sum\limits_{j\ne i}v_{ij}^{L,n},$
it follows that
\[s_{ii}^n-v_{ii}^{L,n}=-\sum\limits_{j\ne i}(s_{ij}^n-v_{ij}^{L,n})\ge 0,\quad (k-1)\gamma_i^n+\frac{\Delta t^n}{m_i}\sum_{j\in\mathcal{I}(S_i),j\ne i}|s_{ij}^n-v_{ij}^n|=k\gamma_i^n.\]
Equality $s_{ii}^n - v_{ii}^{L,n} = 0$ holds only if $s_{ij}^n - v_{ij}^{L,n} = 0$ for all $j \ne i$.
We therefore define
 \[
	U_i^{*,n}:=\begin{cases}
\frac{(k-1)\gamma_i^n}{k\gamma_i^n} U_i^n+\sum\limits_{j\in\mathcal{I}(S_i)
j\ne i}\frac{\frac{\Delta t^n}{m_i}|s_{ij}^n-v_{ij}^{L,n}|}{k\gamma_i^n}U_j^n& \text{if } k\gamma_i^n\ne 0,\\
U_i^n&\text{otherwise},
\end{cases}\]
which is a convex combination of the values $\{U_j^n\}_{j=1}^N$, so $U_i^*\in[U_i^{m,n},U_i^{M,n}]$. Equation \eqref{low-order} can thus be rewritten as
\[	U_{i}^{L,n+1} = U_i^{n}(1-k\gamma_i^n)+k\gamma_i^nU_i^{*,n}.\]

The corresponding high-order update reads
	\begin{align*}
	U^{H,n+1}_i &= U_i^n-\frac{\Delta t^n}{m_i}\sum_{j \in \mathcal{I}(S_i)} (s_{ij}^n - v_{ij}^{H,n}) U_j^n\\
	&=U_i^{L,n+1}+\frac{\Delta t^n}{m_i}\sum_{j\in\mathcal{I}(S_i)}(v_{ij}^{L,n}-v_{ij}^{H,n})(U_i^n-U_j^n)\\
	& = U_i^{L,n+1}+\frac{\Delta t^n}{m_i}\sum_{j\in\mathcal{I}(S_i)}(1-\psi_{ij})v_{ij}^{L,n}(U_i^n-U_j^n).
	\end{align*}
By inserting the expression for $U_i^{L,n+1}$, we obtain
\begin{align*}
	U^{H,n+1}_i &= U_i^{n}(1-k\gamma_i^n)+k\gamma_i^nU_i^{*,n}+\frac{\Delta t^n}{m_i}\sum_{j\in\mathcal{I}(S_i)}(1-\psi_{ij})v_{ij}^{L,n}(U_i^n-U_j^n)\\
	&\le U_i^{n}(1-k\gamma_i^n)+k\gamma_i^nU_i^{M,n}+\frac{\Delta t^n}{m_i}\sum_{j\in\mathcal{I}(S_i)}(1-\psi_{ij})v_{ij}^{L,n}(U_i^n-U_j^n)\\
	&\le U_i^{M,n}+(1-k\gamma_i^n)(U_i^n-U_i^{M,n})+(1-\psi_{i})\gamma_i^{-,n}v_{ij}^{L,n}(U_i^n-U_{i}^{m,n}).
\end{align*}
where the last inequality follows from $\psi_i^n \le \psi_{ij}^n$.
Substituting $U_i^n = \theta_i^nU_i^{M,n}+(1-\theta_i^n)U_i^{m,n}$ into the above inequality leads to
\[	U^{H,n+1}_i \le U_i^{M,n}-(U_i^{M,n}-U_i^{m,n})\left((1-k\gamma_i^n)(1-\theta_i^n)-\theta_i^n(1-\psi_i)\gamma_{i}^{-,n}\right).\]
The other inequality is obtained similarly.

\end{proof}

\subsubsection{Greedy limiting}
According to  \Cref{psi}, the local maximum principle at the level of DOFs is ensured once
\begin{equation}
	(1 - \theta_{i}^{n})(1 -k \gamma_{i}^{n}) - \theta_{i}^{n}(1 - \psi_{i}^{n})\gamma_{i}^{-,n} \ge 0,\quad \theta_{i}^{n} (1 -k \gamma_{i}^{n}) - (1 - \theta_{i}^{n})(1 - \psi_{i}^{n})  \gamma_{i}^{+, n} \ge 0.
\end{equation}
Thus, we can define 
\begin{equation}
	\psi_i^n := \max\left(1-(1 - k\gamma_{i}^{n})\min(\frac{(1 - \theta_{i}^{n})}{\theta_i^n\gamma_i^{-,n}},\frac{\theta_{i}^{n}}{(1-\theta_i^n)\gamma_i^{+,n}}),0\right).
\end{equation}
We select $k=1$ to minimize $\psi_i^n$, then the scheme with the greedy viscosity $V^n:=\Psi^n*V^{L,n}$ satisfies the local maximum principle with respect to DOFs, as established in Lemma \ref{psi}, provided $\gamma_i^n\le 1$. This requirement coincides with the CFL condition~\eqref{CFL} of the scheme employing the minimum viscosity.

\subsubsection{Flux corrected transport (FCT) viscosity}
\label{sec:FCT}
In this subsection, we introduce the flux-corrected transport (FCT) technique \cite{MR1486270, MR534786}, which attains high-order accuracy while preserving bounds at the DOFs.
The central idea is to correct the Galerkin solution using the low-order one so that the midpoint values remain within the admissible range $[U_i^{\min}, U_i^{\max}]$.

Let $U^{L,n+1} \in \mathbb{R}^I$ and $U^{H,n+1} \in \mathbb{R}^I$ denote the low- and high-order solutions, respectively, where the former is defined by (\ref{forward_mass}) and the latter corresponds to the Galerkin solution (\ref{forwardA}).
Subtracting the two expressions gives the relation between the high- and low-order solutions:
\begin{equation}
	U^{H,n+1} = U^{L,n+1} - \Delta t^n M^{-1}V^{L,n}U^n,
\end{equation}
that is, componentwise,
\begin{equation}
\begin{aligned}
	U_i^{H,n+1} =& U_i^{L,n+1}-\frac{\Delta t^n}{m_i}\sum_{j\in\mathcal{I}(S_i)}v^{L,n}_{ij}U_j^{n}\\
 =&U_i^{L,n+1}-\frac{\Delta t^n}{m_i}\sum_{j\in\mathcal{I}(S_i)}v^{L,n}_{ij}(U_j^{n}-U_i^n),
\end{aligned}
\end{equation}
where the last equality follows from  $\sum_{j\in\mathcal{I}(S_i)}v_{ij}^{L,n}=0$. Denote $t_{ij}^n := -v_{ij}^{L,n}(U_j^n-U_i^n)$.
Since it is not guaranteed that $U^{H,n+1}$ satisfies the discrete maximum principle, we introduce a  correction as follows
\begin{equation}
	U_i^{n+1} = U_i^{L,n+1}+\frac{\Delta t^n}{m_i}\sum_{j\in\mathcal{I}(S_i)}l_{ij}^nt_{ij}^n,
\end{equation}
where $l_{ij}^n \in [0,1]$ denote symmetric limiting coefficients,
i.e., $l_{ij}^n = l_{ji}^n$, to be determined.
Because $l_{ij}^n t_{ij}^n = -l_{ij}^n t_{ji}^n$, the corrected solution preserves the conservation property of the scheme, i.e., $\sum_{i=1}^N m_iU_i^{n+1} = \sum_{i=1}^N m_iU_i^{L,n+1}$.
To determine the limiting coefficients $l_{ij}^n$, we first introduce the auxiliary quantities $P^{\pm}_i,Q_i^{\pm}, R_i^{\pm}$
\begin{align}
  &P_i^+:=\sum_{j\in\mathcal{I}(S_i)}\max\{0,t_{ij}^n\},\quad P_i^-:=\sum_{j\in\mathcal{I}(S_i)}\min\{0,t_{ij}^n\},\label{p}\\
  &Q_i^+:=\frac{m_i}{ \Delta t^n}(U_{i}^{\max}-U_{i}^{L,n+1}),\quad Q_i^-:=\frac{m_i}{\Delta t^n}(U_{i}^{\min}-U_{i}^{L,n+1}),\\
    &R_i^+:=\begin{cases}
    \min\{1,\frac{Q_i^+}{P{_i^+}}\}&\text{if}\ P_i^+\ne 0,\\
    1 &\text{otherwise},
  \end{cases}\quad R_i^-:=\begin{cases}
    \min\{1,\frac{Q_i^-}{P{_i^-}}\}& \text{if}\ P_i^-\ne 0,\\
    1& \text{otherwise},\label{r}
  \end{cases}
\end{align}
where the values $U_i^{\min}$ and $U_i^{\max}$ denote the prescribed local bounds of the solution at the DOF $i$.
The limiting coefficients $l_{ij}^n$ are then defined by
 \begin{equation}
l_{ij}^n:=\begin{cases}
    \min\{R_i^+,R_j^-\}& t_{ij}\ge 0,\\
    \min\{R_i^-,R_j^+\}&\text{otherwise}.
  \end{cases}\label{lij}
 \end{equation}
\Cref{thm:maxprinciple} ensures that the modified solution $U^{n+1}$ remains bounded within the prescribed limits $[U_i^{\min}, U_i^{\max}]$.
Rewriting $U_i^{n+1}$, we obtain
	\begin{align*}
	U_i^{n+1} =& U_i^{L,n+1}-\frac{\Delta t^n}{m_i}\sum_{j\in\mathcal{I}(S_i)}l_{ij}^nt_{ij}^n=U_i^{L,n+1}-\frac{\Delta t^n}{m_i}\sum_{j\in\mathcal{I}(S_i)}l_{ij}^nv^{L,n}_{ij}(U_j^{n}-U_i^n)\\
	=& U_i^{n}-\frac{\Delta t^n}{m_i}\sum_{j\in\mathcal{I}(S_i)} \left(s_{ij}^{n}-(1-l_{ij}^n)v_{ij}^{L,n}\right)U_j^{n}.
	\end{align*}
	This formulation is thus equivalent to choosing the high-order viscosity as
\begin{equation}
	V^{H,n} = \Psi^n *V^{L,n},\quad \Psi^n = E-L^n.
\end{equation}

\subsection{Inflow boundary condition}\label{sec:Inflow}
We now consider the treatment of inflow boundary conditions. The scheme (\ref{forward}) can be rewritten as
\begin{equation}
	U^{n+1} = (I-\Delta t^n M^{-1}S^n)U^n+L^n,\label{forward_inflow}
\end{equation}
where $S^n$ is defined as before, and the $i$-th component $l_i^n$ of $L^n$ is given by $l_h(u_h^n,\varphi_i)$ from (\ref{lh}).
To obtain the low-order scheme, we simply neglect the inflow contribution $L^n$ and proceed as in \Cref{sec:viscosity}. For the high-order scheme, however, the term $L^n$ cannot be omitted, and the greedy viscosity cannot be directly applied. Instead, we propose a modified limiting procedure inspired by the FCT technique.
Similar to \Cref{sec:FCT}, we take the Galerkin solution as the high-order solution to be limited:
\begin{equation}U^{H,n+1}_i=U_i^{L,n+1}+\frac{\Delta t^n}{m_i}\sum_{j\in\mathcal{I}(S_i)}t_{ij}^n+\frac{\Delta t^n}{m_i}l_i^n.\end{equation}
We now limit both the numerical flux term and the inflow contribution by introducing the modified update
\begin{equation}
\begin{aligned}
	&	U_i^{n+1}=U_i^{L,n+1}+\frac{\Delta t^n}{m_i}\sum_j l_{ij}^nt_{ij}+\frac{\Delta t^n}{m_i}\alpha_i^n l_i^n,
\end{aligned}
\end{equation}
where $l_{ij}$ is defined though (\ref{p}), (\ref{r}) and (\ref{lij}) with 
\begin{equation}
	Q_i^-:=\frac{m_i}{\Delta t^n}(U_{i}^{\min}-U_i^{L,n+1})-\alpha_i^n l_i^n,\quad Q_i^+:=\frac{m_i}{\Delta t^n}(U_{i}^{\max}-U_i^{L,n+1})-\alpha_i^n l_i^n\label{q}
\end{equation}
and the inflow limiter
\begin{equation}
	\alpha_i^n:=\begin{cases}\min\{1,\max\{0,\frac{m_i}{\Delta t}\frac{U_{i}^{\max}-U_i^{L,n+1}}{l_i^n},\frac{m_i}{\Delta t}\frac{U_{i}^{\min}-U_i^{L,n+1}}{l_i^n}\}\}&\text{if }l_i\ne 0,\\1&\text{otherwise}.\end{cases}.\label{alpha}
\end{equation}

From the definition (\ref{alpha}), one observes that a smaller time step $\Delta t^n$ allows $\alpha_i^n$ to approach $1$, leading to a more relaxed limitation.
\begin{theorem}[Discrete maximum principle with respect to DOFs]\label{thm:maxprinciple}
Assume that the low-order solution satisfies 
$U_i^{\min}\le U_i^{L,n+1}\le U_i^{\max}$,
then after applying the above limiting procedure, the updated solution $U_i^{n+1}$ remains within the same bounds, i.e., $
U_i^{\min}\le U_i^{n+1}\le U_i^{\max}.$

In particular, when $l_i^n=0$ for all $i=1,2,\dots,N$, the same result holds for the FCT viscosity introduced in \Cref{sec:FCT}.
\end{theorem}

\begin{proof}
From the definitions \eqref{p}-\eqref{r} and \eqref{alpha}, we have
\begin{align*}
U_i^{n+1}-U_i^{\min}
&=U_i^{L,n+1}-U_i^{\min}+\frac{\Delta t^n}{m_i}\left(\sum_{j\in\mathcal{I}(S_i)}l_{ij}^nt_{ij}^n+\alpha_i^n l_i^n\right)\\
&\ge U_i^{L,n+1}-U_i^{\min}+\frac{\Delta t^n}{m_i}\left(R_i^-\sum_{j\in\mathcal{I}(S_i)}\min(0,t_{ij}^n)+\alpha_i^n l_i^n\right)\\
&\ge U_i^{L,n+1}-U_i^{\min}+\frac{\Delta t^n}{m_i}\left(\frac{Q_i^-}{P_i^-}P_i^-+\alpha_i^n l_i^n\right)\\
&=U_i^{L,n+1}-U_i^{\min}+\frac{\Delta t^n}{m_i}\left(Q_i^-+\alpha_i^n l_i^n\right)\ge 0.
\end{align*}
Similarly, we obtain
\[U_i^{n+1}-U_i^{\max}\le U_i^{L,n+1}-U_i^{\max}+\frac{\Delta t^n}{m_i}\alpha_i^n l_i^n+\frac{\Delta t^n}{m_i}Q_i^+\le 0.\]
Hence, the discrete maximum principle holds.
\end{proof}

\section{Reconstruction of the solution on $\frac{h}2$-mesh}\label{sec:postprocess}
In the previous sections, we have developed several schemes that can preserve maximum principle with respect to DOFs, i.e., $u_{\min}\le \min_{i}U_i\le \max_iU_i\le u_{\max}$, where $u_{\min}$ and $u_{\max}$ denote the lower and upper bounds of the initial condition, respectively.  
Since the cell average is computed as the mean of the three midpoint values of its edges, it follows immediately that the cell averages inherit the same bounds as $\{U_i\}_{i=1}^N$.  Furthermore, the following theorem tell us that a solution satisfies this kind of maximum principle has a global lower bound $2u_{\min}-u_{\max}$ and upper bound $2u_{\max}-u_{\min}$.
	
\begin{theorem}[Global bound]
	Suppose that $u_h=\sum_{i=1}^NU_i\varphi_i\in \mathcal{V}_h$ with  $u_{\min}\le \min_{i}U_i\le \max_iU_i\le u_{\max}$, then 
	\begin{equation}
		2u_{\min}-u_{\max}\le u_h(\mathbf{x})\le 2u_{\max}-u_{\min}.
	\end{equation}
\end{theorem}
\begin{proof}
	The solution on cell $K\in \mathcal{K}_h$ can be expressed as 
\[
	u_h|_K(x) = \sum_{s=1}^3U_{c_K(s)}\varphi_{c_K(s)}(x)=\sum_{s=1}^3U_{c_K(s)}(1-\lambda_s),
\]
where $\lambda_s$ denotes the barycentric coordinate of $\mathbf{x}$ in $K$ and $c_K(s)$ is the global index corresponding to the local vertex $s$.  
Since $u_h|_K$ is affine in $\mathbf{x}$, its extrema are attained at the vertices of $K$.  
Assume, without loss of generality, that the extremum occurs at $\mathbf{x}_{c_K(1)}$, where $\lambda_1=1$ and $\lambda_2=\lambda_3=0$.  
Then $u_h|_K(\mathbf{x}_{c_K(1)}) = -U_{c_K(1)} + U_{c_K(2)} + U_{c_K(3)}$.
Using $u_{\min}\le U_i\le u_{\max}$ for all $i$, we obtain
\[
2u_{\min}-u_{\max}\le u_h|_K(\mathbf{x})\le 2u_{\max}-u_{\min}, 
\qquad \forall K\in\mathcal{K}_h,
\]
which completes the proof.
\end{proof}

Unfortunately, a sharper bound cannot be obtained. Consider the discontinuous function $u(\mathbf{x})$ defined on $[-1,1]^2$ by
$u(\mathbf{x})=\begin{cases}
	1&\text{if}\  x>0,\\
	-1&\text{otherwise}.
\end{cases}$
We construct a uniform triangular mesh over $[-1,1]^2$ with 2k + 1 ($k\in\mathbb{N}_+$) grid points in each spatial direction. For all midpoints on the left side of the domain, the CR interpolation $u_h(\mathbf{x})$ takes the value $-1$ while for the remaining midpoints it takes the value $1$. Now, consider the mesh element adjacent to the line $x=0$, whose one edge lies on $x=0$ and the other two edges lie on the left side; see \Cref{fig:counter}. It is straightforward to verify that the minimum value of $u_h(\mathbf{x})$ on this element is $-3$, so the lower bound $2u_{\min}-u_{\max}$ is attained. Similarly, we can construct another example to reach the upper bound $2u_{\max}-u_{\min}$.
This indicates that the CR interpolation may exhibit non-decaying oscillations near discontinuities, regardless of mesh refinement. This behavior is analogous to the Runge phenomenon in polynomial interpolation and the Gibbs phenomenon in Fourier series expansions.

\begin{figure}[htbp]
    \centering
    \begin{minipage}[t]{0.48\textwidth}
        \centering
        \resizebox{\linewidth}{!}{%
        \begin{tikzpicture}[scale=4.55]
        
            \draw (0,0) rectangle (1,1);
        \node[left] at (0,0) {$-1$};
        \node[right] at (1,0) {$1$};
        \node[left] at (0,1) {$1$};
        
        \foreach \i in {0.25,0.75} {
            \draw (\i,0) -- (\i,1);
            \draw (0,\i) -- (1,\i);
        }
        \draw (0,0.5) -- (1,0.5);

        \foreach \x in {0.25,0.75} {
            \foreach \y in {0,0.25,0.5,0.75} {
                \draw (\x,\y) -- (\x+0.25,\y+0.25);
            }
        }
        \foreach \x in {0,0.5} {
            \foreach \y in {0.25,0.75} {
                \draw (\x,\y) -- (\x+0.25,\y+0.25);
            }
        }
        \foreach \x in {0,0.5} {
            \foreach \y in {0,0.5} {
                \draw (\x,\y) -- (\x+0.5,\y+0.5);
            }
        }
        
        \draw[red,thick] (0.5,-0.1) -- (0.5,1.1);
        \node[red,above] at (0.5,1.1) {$x=0$};
        \fill[blue!15] (0.25,0.5) -- (0.5,0.5) -- (0.5,0.75) -- cycle;
        \draw[blue] (0.375,0.5) circle (0.4pt);
        \draw[blue] (0.5,0.625) circle (0.4pt);
        \draw[blue] (0.375,0.625) circle (0.4pt);
        \node[blue,below] at (0.375,0.5) {$-1$};
        \node[blue,right] at (0.5,0.625) {$1$};
        \node[blue,above left] at (0.375,0.625) {$-1$};

        \filldraw[red] (0.25,0.5) circle (0.4pt);
        \node[red,below left] at (0.25,0.6) {$-3$};
        \makeatletter
        \pgfresetboundingbox
        \path[use as bounding box] (-0.35,-0.35) rectangle (1.35,1.35);
        \makeatother
        \end{tikzpicture}%
        }\\[0.08\baselineskip]
        \refstepcounter{figure}\label{fig:counter}
        \par\small\sloppy\noindent\begingroup
        \hangafter=1\relax\hangindent=2.4em\relax
        \noindent\textbf{\figurename~\thefigure.} The CR interpolation attains a value of $-3$ at the lower-left vertex of the blue triangle.
        \par\endgroup
    \end{minipage}
    \hfill
    \begin{minipage}[t]{0.48\textwidth}
        \centering
        \resizebox{\linewidth}{!}{%
        \begin{tikzpicture}[scale=4.55]
        
        \draw (0,0) rectangle (1,1);

        \fill[red!15] (0.25,0.25) -- (0.5,0.25) -- (0.75,0.5) -- (0.75,0.75) -- (0.5,0.75) -- (0.25,0.5) -- cycle;

        \foreach \i in {0.25,0.75} {
            \draw[dashed] (\i,0) -- (\i,1);
            \draw[dashed] (0,\i) -- (1,\i);
        }
        \draw (0.5,0) -- (0.5,1);
        \draw (0,0.5) -- (1,0.5);  

        \foreach \x in {0.25,0.75} {
            \foreach \y in {0,0.25,0.5,0.75} {
                \draw[dashed] (\x,\y) -- (\x+0.25,\y+0.25);
            }
        }
        \foreach \x in {0,0.5} {
            \foreach \y in {0.25,0.75} {
                \draw[dashed] (\x,\y) -- (\x+0.25,\y+0.25);
            }
        }
        \foreach \x in {0,0.5} {
            \foreach \y in {0,0.5} {
                \draw (\x,\y) -- (\x+0.5,\y+0.5);
            }
        }
        
        \foreach \x in {0.25,0.75} {
            \foreach \y in {0,0.25,0.5,0.75,1} {
                \draw[blue] (\x,\y) circle (0.4pt);
            }
        }
        \foreach \x in {0,0.5,1} {
            \foreach \y in {0.25,0.75} {
                \draw[blue] (\x,\y) circle (0.4pt);
            }
        }

        \node[left] at (0.25,0.25) {$\mathbf{y}_6$};
        \node[below left] at (0.5,0.25) {$\mathbf{y}_1$};
        \node[below right] at (0.75,0.5) {$\mathbf{y}_2$};
        \node[right] at (0.75,0.75) {$\mathbf{y}_3$};
        \node[above right] at (0.5,0.75) {$\mathbf{y}_4$};
        \node[above left] at (0.25,0.5) {$\mathbf{y}_5$};
        \node[below right] at (0.5,0.5) {$\mathbf{y}_0$};
        \makeatletter
        \pgfresetboundingbox
        \path[use as bounding box] (-0.35,-0.35) rectangle (1.35,1.35);
        \makeatother
        \end{tikzpicture}%
        }\\[0.08\baselineskip]
        \refstepcounter{figure}\label{fig:square_partition}
        \par\small\sloppy\noindent\begingroup
        \hangafter=1\relax\hangindent=2.4em\relax
        \noindent\textbf{\figurename~\thefigure.} The $h$-mesh (solid lines), the $\tfrac{h}{2}$-mesh (dashed lines), and the polygon around $\mathbf{y}_0$ (red region). Blue points denote edge midpoints of the $h$-mesh.
        \par\endgroup
    \end{minipage}
\end{figure}

To eliminate these oscillations, we introduce a post-processing technique designed to smooth the solution. The main idea is to reconstruct a new continuous $P_1$ solution on a finer mesh from the CR solution on the original mesh. 
The finer mesh here is  obtained by subdividing each triangle into four sub-triangles via connecting the midpoints of its edges. We refer to this refined mesh as the  $\frac{h}2$-mesh, while the original mesh is called the $h$-mesh; see \Cref{fig:square_partition}.



The Wachspress coordinates, introduced by Wachspress in the 1975 \cite{MR426460}, are a type of generalized barycentric coordinates defined over convex polygon.  Let $T$ be an $s$-sided convex polygon with vertices $\{\mathbf{y}_i\}_{i=1}^s$, the Wachspress coordinates can be defined by the formula \begin{equation}
	w_i(\mathbf{x}):=\frac{\tilde{w}_i(\mathbf{x})}{\sum_{j=1}^s\tilde{w}_j(\mathbf{x})},\quad \tilde{w}_i(\mathbf{x})=\frac{\mathcal{A}(\mathbf{y}_{i-1},\mathbf{y}_i,\mathbf{y}_{i+1})}{\mathcal{A}(\mathbf{x},\mathbf{y}_{i-1},\mathbf{y}_i)\mathcal{A}(\mathbf{x},\mathbf{y}_{i-1},\mathbf{y}_i)},
\end{equation}
where $\mathcal{A}(\mathbf{a},\mathbf{b},\mathbf{c}):=\frac{1}2\left|\begin{matrix}
1 & 1 & 1 \\
a_1 & b_1 & c_1 \\
a_2 & b_2 & c_2
\end{matrix}\right|$ denotes the signed area of the triangle formed by points $\mathbf{a},\mathbf{b},\mathbf{c}$.
The Wachspress coordinates possess the properties of non-negativity, partition of unity and linear precision, that is, for any $\mathbf{x}\in T$, 
$0\le w_i(\mathbf{x})\le 1 \ (i=1,2,\cdots, s)$, $\sum_{i=1}^sw_i(\mathbf{x})=1$ and $l(\mathbf{x})=\sum_{i=1}^sw_i(\mathbf{x})l(\mathbf{y}_i)$, for  any linear function  $l(\mathbf{x})$ defined $\overline{T}$. In general, the polygon obtained by connecting all the neighboring vertices on the $\frac{h}2$-mesh is star-shaped. Although such polygon may not be convex,  this issue can be handled using mean value coordinates \cite{MR3218572} instead. Since  the uniform meshes employed in our numerical tests satisfy the convexity assumption, we will focus on the Wachspress coordinates for simplicity.

We now specify the point values for all vertices on the $\frac{h}2$-mesh as follows:
\begin{enumerate}
\item Midpoints of edges on the $h$-mesh: take the CR solution values at these points.
\item Vertices of the $h$-mesh with more than two adjacent edge midpoints: compute the Wachspress coordinates of these vertices with respect to the convex polygon formed by connecting their neighboring vertices on the $\frac{h}2$-mesh (see \Cref{fig:square_partition}), and use these coordinates as combination coefficients to evaluate the interpolated value.
\item Remain points: take the CR solution values  or calculate from the boundary value function if it exists, since all the remain points lies on the boundary. 
\end{enumerate}
We denote the above operation as $\mathcal{I}^W:\mathcal{V}_h\rightarrow\mathcal{W}_{h/2}:=\{v_h\in P(\mathcal{K}_{h/2}):v_h\in C^0(\Omega)\}$, where $\mathcal{K}_{h/2}$ is the triangulation of the $\frac{h}2$-mesh.

	In our numerical experiments, we only reconstruct the solution at the final time since at each time step only the degrees of freedom of the CR element are needed.

\begin{theorem}[Discrete maximum principle on the whole domain]{}
	Suppose $u_h=\sum_{i=1}^NU_i\varphi_i\in \mathcal{V}_h$ satisfies that
	$u_{\min}\le \min_{i}U_i\le \max_iU_i\le u_{\max}$, then the reconstructed $P_1$ solution $\mathcal{I}u_h\in \mathcal{W}_{h/2}$ preserves this bound on the whole domain, i.e.,
	\begin{equation}
		u_{\min}\le \mathcal{I}(u_h)\le u_{\max}.
	\end{equation}
Here for simplicity, we further assume that the boundary values of $\mathcal{I}^W(u_h)$ are all computed by a given boundary function $u|_{\p\Omega}$ satisfying $u_{\min}\le u|_{\p\Omega}\le u_{\max}$.

\end{theorem}

\begin{proof}
	Under the assumptions, it is easy to see that the values of $\mathcal{I}(u_h)$ at all vertices of the $\frac{h}2$-mesh are within the interval $[u_{\min},u_{\max}]$. Since $\mathcal{I}(u_h)$ is a Lagrange $P_1$ function, the extrema of $\mathcal{I}(u_h)$ on each cell are necessarily found among the three vertices of the element. Thus we have the conclusion. 
\end{proof}

The discrete maximum principle over the entire domain for the scalar conservation law is crucial in coupled problems, where this solution acts as a coefficient in other equations that rely on boundedness for well-posedness—for example, the density must remain positive in compressible flow.

To establish an accuracy estimate for the reconstructed solution,  we  partition the domain into two types of regions associated with the  $\frac{h}2$-mesh. Let $\mathcal{J}_h\subset\mathcal{K}_{\frac{h}2}$ 
 denote the subset of triangles whose three vertex values are directly taken from the CR solution $u$, and let $\mathcal{T}_h$ denote the set of convex polygons,
each of which is formed by connecting the neighboring edge midpoints surrounding an interior vertex of the $h$-mesh.
\begin{lemma}[Error estimate of the reconstruction]
	Assume that $u\in H^2(\Omega)$,  $\mathcal{I}^{CR}:H^2(\Omega)\rightarrow\mathcal{V}_h$ is the CR interpolation operator satisfies $\mathcal{I}^{CR}(u)(\mathbf{x}_i)=u(\mathbf{x}_i)$, $\forall\ i=1,2,\cdots , N$. Then we have
	\begin{equation}
		\|u-\mathcal{I}^W(\mathcal{I}^{CR}(u))\|_{L^2(\Omega)}\le Ch^2|u|_{H^2(\Omega)},
	\end{equation}
where $C$ is a constant  depends on the  maximum number of sides among the polygons in the $\mathcal{T}_h$.

	Here, for  simplicity, the boundary values of $\mathcal{I}^W(\mathcal{I}^{CR}u)$ are taken directly from the original function $u$. 
\end{lemma}
\begin{proof}
	Since  $H^2(\Omega)\hookrightarrow C^0(\Omega)$,  $u$ has a continuous version and the CR interpolation $\mathcal{I}^{CR}(u)$ is well-defined.
	By  the local interpolation error estimate in \cite[Theorem 11.13]{MR4242224}, we obtain
	\begin{equation}\norm{u-\mathcal{I}^{W}(\mathcal{I}^{CR}(u))}_{L^2(K)}\le C_1h^2|u|_{H^2(K)},\quad \forall\ K\in \mathcal{J}_h,\label{onJ_h}\end{equation}
	where the constant $C_1$ is independent of $h$.

	For each polygonal $T\in\mathcal{T}_h$, denote the vertices as $\{\mathbf{y}_j\}_{j=1}^s$ (where $s$ is the number of sides of $T$), and denote by $\{K_j\}_{j=1}^s$ the triangulation of $T$ formed by connecting all $\mathbf{y}_j$ to the interior vertex $\mathbf{y}_0$ of $T$: $
	K_i=\triangle \mathbf{y}_0\mathbf{y}_j\mathbf{y}_{j+1}$ with the convention $\mathbf{y}_{s+1}=\mathbf{y}_1$.
	Define a finite dimensional space
	\[Q(T):=\{v_h|_{K_j}\in P_1(K_j):v_h\in C^0(T), v_h(\mathbf{y}_0)=\sum_{j=1}^sw_jv_h(\mathbf{y}_j)\}	
\supsetneqq P_1(T),\]
	where $\{w_j\}_{j=1}^s$ are the Wachspress coordinates of $\mathbf{y}_0$ with respect to the the vertices $\{\mathbf{y}_j\}_{j=1}^s$. It's easy to prove that $\mathrm{dim}\ Q(T)=s$  with basis functions $\{\theta_i\}_{i=1}^s$ defined by
	\[\theta_i(\mathbf{y}_j)=\delta_{ij},\quad \theta_i(\mathbf{\mathbf{y}_0})=w_i, \quad i,j=1,2,\cdots, s.\]
	For $v\in C^0(T)$, define the local interpolation operator $\mathcal{I}^{Q(T)}:C^0(T)\rightarrow Q(T)$  by 
	\[\mathcal{I}^{Q(T)}(v)(\mathbf{y}_j)=v(\mathbf{y}_j),\  j=1,2,\cdots,s,\quad \mathcal{I}^{Q(T)}(v)(\mathbf{y}_0)=\sum_{j=1}^sw_jv(\mathbf{y}_j).\]
	Clearly, $P_1(T)$ is pointwise invariant under $\mathcal{I}^{Q(T)}$. 
	Following the same reasoning as in \cite[Theorem 11.13]{MR4242224}, we obtain
	\begin{equation}\norm{u-\mathcal{I}^{W}(\mathcal{I}^{CR}(u))}_{L^2(T)}=\norm{u-\mathcal{I}^{Q(T)}(u)}_{L^2(T)}\le C_2h^2|u|_{H^2(T)},\quad \forall\ T\in \mathcal{T}_h,\label{onT_h}\end{equation}
	where $C_2$ is independent of $h$, but is relative to the number of sides of the polygon $T$.  

	Combined (\ref{onJ_h}) and (\ref{onT_h}), we have
	\[\norm{u-\mathcal{I}^{W}(\mathcal{I}^{CR}(u))}_{L^2(\Omega)}\le Ch^2|u|_{H^2(\Omega)}, \]
	where $C$ is independent of $h$ but depends on the  maximum number of sides among the polygons in the $\mathcal{T}_h$.

\end{proof}





\begin{theorem}[Second-order accuracy]
	Assume that $u\in H^2(\Omega)$, $u_h\in \mathcal{V}_h$ is a numerical solution satisfying $\norm{u-u_h}_{L^2(\Omega)}\le Ch^2|u|_{H^2(\Omega)}$, then the reconstructed solution $\mathcal{I}^W(u_h)\in \mathcal{W}_{\frac{h}2}$ also has second-order accuracy, i.e.,
	\[\norm{u-\mathcal{I}^W(u_h)}_{L^2(\Omega)}\le Ch^2|u|_{H^2(\Omega)},\]
	where $C$ is a constant  depends on the  maximum number of sides among the polygons in the $\mathcal{T}_h$.

	For simplicity, we  assume that the boundary values of $\mathcal{I}^W(u_h)$ are all taken from the solution $u$ (or the inflow function).
\end{theorem}
\begin{proof}
	Firstly, we apply the triangle inequality and obtain
	\begin{equation}
		\norm{u-\mathcal{I}^{W}(u_h)}_{L^2(\Omega)}\le\norm{u-\mathcal{I}^W(\mathcal{I}^{CR}(u))}_{L^2(\Omega)}+\norm{\mathcal{I}^W(\mathcal{I}^{CR}(u))-\mathcal{I}^W(u_h)}_{L^2(\Omega)}\label{error}.
	\end{equation}
	The first term on the right hand side of (\ref{error}) has been estimated in the previous lemma.
	
	Secondly, we want to prove that
	\[\norm{\mathcal{I}^W(v_h)}_{L^2(\Omega-\cup_{K\in\mathcal{D}_{\frac{h}2}}K)}\le C\norm{v_h}_{L^2(\Omega)},\]
	where $\mathcal{D}_{\frac{h}2}$ consists of all triangles that have  all triangles in the triangulation $\mathcal{K}_{\frac{h}2}$ that intersect  $\partial \Omega$.
	For any $T\in\mathcal{T}_h$,
	\begin{align*}
		\int_T(\mathcal{I}^W(v_h))^2\, \mathrm{d}\mathbf{x}=&\sum_{j=1}^s\int_{K_j}(\mathcal{I}^W(v_h))^2\, \mathrm{d}\mathbf{x}\\
		\le& \sum_{j=1}^s\frac{|K_j|}3\left(\left(v_h(\mathbf{y}_j)\right)^2+\left(v_h(\mathbf{y}_{j+1})\right)^2+\left(v_h(\mathbf{y}_0)\right)^2\right),
	\end{align*}
	Note that
	\begin{align*}
		\left(v_h(\mathbf{y}_0)\right)^2=\left(\sum_{j=1}^sw_jv_h(\mathbf{y}_j)\right)^2\le \left(\sum_{j=1}^sw_j^2\right)\left(\sum_{j=1}^sv_h(\mathbf{y}_j)^2\right)\le s\sum_{j=1}^s \left(v_h(\mathbf{y}_j)\right)^2,
	\end{align*}
	we have
	\begin{align*}
		\sum_{T\in \mathcal{T}_h}\int_T(\mathcal{I}^W(v_h))^2\, \mathrm{d}\mathbf{x}\le C\sum_{K\in\mathcal{K}_h}\frac{|K|}3\sum_{i=1}^3v_{c_K(i)}^2\le C\norm{v_h}_{L^2(\Omega)}^2.
	\end{align*}
Here and throughout the proof, the letter $C$ denotes a generic positive constant that may change from line to line, but is independent of $h$.
	Therefore, 
	\begin{align*}\norm{\mathcal{I}^{W}(v_h)}_{L^2(\Omega-\cup_{K\in\mathcal{D}_{\frac{h}2}}K)}^2=&\sum_{K\in\mathcal{J}_h, K\notin\mathcal{D}_{\frac{h}2}}\norm{\mathcal{I}^{W}(v_h)}_{L^2(K)}^2+\sum_{T\in\mathcal{T}_h}\norm{\mathcal{I}^{W}(v_h)}_{L^2(T)}^2\\
	\le&\sum_{K\in\mathcal{J}_h}\norm{v_h}_{L^2(K)}^2+\sum_{T\in\mathcal{T}_h}\norm{v_h}_{L^2(T)}^2
\le C \norm{v_h}_{L^2(\Omega)}^2,\end{align*}
Now we estimate the second term of (\ref{error}) :
	\begin{align*}
		&\norm{\mathcal{I}^W(\mathcal{I}^{CR}(u))-\mathcal{I}^W(u_h)}_{L^2(\Omega)}\\=&\norm{\mathcal{I}^W(\mathcal{I}^{CR}(u))-\mathcal{I}^W(u_h)}_{L^2(\cup_{K\in\mathcal{D}_{\frac{h}2}})} +\norm{\mathcal{I}^W(\mathcal{I}^{CR}(u))-\mathcal{I}^W(u_h)}_{L^2(\Omega-\cup_{K\in\mathcal{D}_{\frac{h}2}})}\\
		\le &C\norm{\mathcal{I}^{CR}(u)-u_h}_{L^2(\Omega)}+ \norm{\mathcal{I}^W\left(\mathcal{I}^{CR}(u)-\mathcal{I}^W(u_h)\right)}_{L^2(\Omega-\cup_{K\in\mathcal{D}_{\frac{h}2}})}\\
		\le &C\norm{\mathcal{I}^{CR}(u)-u_h}_{L^2(\Omega)}\\
		\le&C\norm{\mathcal{I}^{CR}(u)-u}_{L^2(\Omega)}+C\norm{u-u_h}_{L^2(\Omega)}\\
		\le &Ch^2|u|_{H^2(\Omega)}.
	\end{align*}

	Finally, combining all the above estimates yields the desired result.
\end{proof}

\section{Numerical tests}
\label{sec:numericaltests}
We present several numerical tests to verify the maximum-principle-preserving methods introduced above. All tests in this section use  uniform triangular meshes, and the low-order viscosity is the minimum viscosity.  The time-stepping method is the SSP RK(3,3) scheme described in Section \ref{sec:timediscretization}, combined with an automatic time-step control:

\begin{enumerate}
\item At each stage of the SSP RK(3,3), check whether the time step $\Delta t^n$ satisfies the CFL condition (\ref{CFL}).
\item If the condition is satisfied, proceed with the time advancement.
\item Otherwise, reduce the time step as $\Delta t^n \gets c\Delta t^n$, and restart from the first stage.
\end{enumerate}

In all experiments, the default value of the reduction factor is $c=0.5$, and the final time is set to $t=1$ unless otherwise specified. All tests are implemented using the MFEM library \cite{mfem-web}.

\subsection{Smooth rotation}
We first consider a smooth rotation problem on $\Omega = [-1,1]^2$ with velocity field   $\vec{\beta}=2\pi(-y,x)^T$ and the initial condition be
\begin{equation}u_0(\bm{x})=\frac12(1-\tanh(\frac{(\bm{x}-\bm{x}_0)^2}{r_0^2}-1)),\quad r_0=0.25, \bm{x}_0=(0.3,0),\end{equation}
so that the exact solution is given by
\begin{equation}
	u(t,x,y)=u_0(x\cos(2\pi t)+y\sin(2\pi t),-x\sin(2\pi t)+y\cos(2\pi t)).
\end{equation}

We test five mesh sizes $h = 0.1, 0.05, 0.025,0.0125$ and $0.00625$. Results are reported in  \Cref{test1}. In the first method (rows 2-5), we use greedy viscosity as the higher-order viscosity. In the second method (rows 6-9, labeled ``Loc. FCT visc.''), we define the local minimum $U_i^{\min}$ and the local maximum $U_i^{\max}$ by $U_i^{\min}=\min_{j\in\mathcal{I}(S_i)}\{U_j^{L,n+1}\}$ and $U_i^{\max}=\max_{j\in\mathcal{I}(S_i)}\{U_j^{L,n+1}\}$ in the FCT viscosity. In the third method (rows 10-13, labeled ``Glob. FCT visc.''), we define $U_i^{\min}=\min_{\mathbf{x\in\Omega}}u_0(\mathbf{x})$ and $U_i^{\max}=\max_{\mathbf{x\in\Omega}}u_0(\mathbf{x})$ in the FCT viscosity. And the last two rows come from \cite{MR3171280} with continuous ${P}_1$ element and entropy viscosity. All methods achieve optimal convergence rates in $L^2$-norm, and the CR solution together with its reconstruction are almost have the same behaviour. Our methods  on the $h$-meshes obtain results are comparable to that from continuous ${P}_1$ entropy viscosity method on the $\frac{h}2$-meshes but with fewer degrees of freedom.

\begin{table}[htbp]
    \centering
	\caption{Smooth rotation, convergence rates:  CR element (greedy/local FCT/global FCT viscosity) vs. continuous $P_1$ element (entropy viscosity).}
\label{test1}
    \resizebox{0.65\textwidth}{!}{
\begin{tabular}{ccccccc}
\toprule
     & h &0.1& 0.05 & 0.025 & 0.0125 &0.00625\\
\midrule
\multirow{2}{7em}{Greedy visc.} & $L^2$ & 6.94E-02 & 9.14E-03 & 1.36E-03 & 2.48E-04 & 4.86E-05\\
& Rate & -- & 2.92 & 2.75 & 2.45 &2.35 \\
Reconstruction & $L^2$ & 7.17E-02 & 8.77E-03 & 1.44E-03 & 2.79E-04 & 5.87E-05\\
& Rate & -- & 3.03 & 2.61 & 2.36 &2.25 \\
\midrule
\multirow{2}{7em}{Loc. FCT visc.} & $L^2$ & 4.34E-02 & 7.34E-03 & 1.01E-03 & 1.98E-04 & 4.13E-05\\
& Rate & -- & 2.56 & 2.86 & 2.35 &2.26 \\
Reconstruction & $L^2$ & 4.61E-02 & 7.44E-03 & 1.20E-03 & 2.51E-04 & 5.43E-05\\
& Rate & -- & 2.63 & 2.63 & 2.26 &2.21 \\
\midrule
\multirow{2}{7em}{Glob. FCT visc.} & $L^2$ & 2.03E-02 & 5.49E-03 & 7.04E-04 & 1.50E-04 & 3.34E-05\\
& Rate & -- & 1.89 & 2.96 & 2.23 &2.17 \\
Reconstruction & $L^2$ &2.45E-02 & 6.20E-03 & 9.91E-04 & 2.19E-04 & 4.90E-05\\
& Rate & -- & 1.98 & 2.65 & 2.18 &2.16 \\
\midrule
\multirow{2}{7em}{continuous ${P}_1$ entropy visc.} & $L^2$ & 1.48E-01 & 4.86E-02 & 1.09E-02 & 1.73E-03 & 2.50E-04\\
& Rate & -- & 1.67 & 2.20 & 2.68 &2.81 \\
\bottomrule
\end{tabular}
}
\end{table}

\subsection{Swirling deformation flow}
We consider the domain $\Omega = [0,1]^2$ with the time-dependent swirling velocity field
\begin{equation} \vec{\beta}(\bm{x},t):=\begin{pmatrix}
	-2\sin(\pi y)\cos(\pi y)\sin^2(\pi x)\cos(\pi t)\\
	2\sin(\pi x)\cos(\pi x)\sin^2(\pi y)\cos(\pi t)
\end{pmatrix},\end{equation}
which represents a swirling deformation flow \cite{MR3738312}.
The initial condition is
\begin{equation}u_0(\bm{x})=\sin(2\pi x)\sin(2\pi y),\end{equation}
and the motion is periodic with period 2. The solution is smooth in both space and time, returning to the initial data at every half period, i.e., $u(\bm{x}, k) = u_0(\bm{x})$ for all $k\in\mathbb{N}$.

We test five meshsizes $h = 0.05, 0.025, 0.0125,0.00625$ and $0.003125$, with the results summarized in \Cref{test2}. The method with global FCT viscosity achieves second-order accuracy, and the errors on the $h$-meshes are nearly identical to those from the continuous  ${P}_1$ element with global FCT viscosity \cite{MR3738312} on the $\frac{h}2$-meshes. The methods with greedy viscosity and local FCT viscosity are slightly worse than the above two methods, but they still have convergence rates more than $1.5$. This is likely due to the local maximum principle enforced by these methods, which prevents the solution from approaching the high-order interpolant too closely.
\begin{table}[htbp]
    \centering
	\caption{Swirling deformation flow,  convergence rates:  CR element (greedy/local FCT/global FCT viscosity) vs. continuous $P_1$ element (FCT viscosity).}
\label{test2}
    \resizebox{0.65\textwidth}{!}{
\begin{tabular}{ccccccc}
\toprule
     & h  & 0.05 & 0.025 &0.0125 &0.00625& 0.003125\\
\midrule
\multirow{2}{7em}{Greedy visc.} & $L^2$ & 7.37E-02 & 2.00E-02 & 5.27E-03 & 1.59E-03 & 5.33E-04\\
& Rate & -- & 1.88 & 1.93 & 1.73 &1.58 \\
Reconstruction & $L^2$ & 6.67E-02 & 1.70E-02 & 3.85E-03 & 9.68E-04 & 2.90E-04\\
& Rate & -- & 1.97 & 2.14 & 1.99 &1.74 \\
\midrule
\multirow{2}{7em}{Loc. FCT visc.} & $L^2$ & 4.80E-02 & 1.37E-02 & 4.36E-03 & 1.49E-03 & 5.23E-04\\
& Rate & -- & 1.81 & 1.66 & 1.55 & 1.51 \\
Reconstruction & $L^2$ & 3.88E-02 & 9.48E-03 & 2.55E-03 & 8.02E-04 & 2.73E-04\\
& Rate & -- & 2.03 & 1.90 & 1.67 &1.55 \\
\midrule
\multirow{2}{7em}{Glob. FCT visc.} & $L^2$ & 4.65E-03 & 1.07E-03 & 2.58E-04 & 6.40E-05 & 1.60E-05\\
& Rate & -- & 2.12 & 2.05 & 2.01 &2.00 \\
Reconstruction & $L^2$ &6.55E-03 & 1.55E-03 & 3.83E-04 & 9.58E-05 & 2.40E-05\\
& Rate & -- & 2.08 & 2.01 & 2.00 &2.00 \\
\midrule
\multirow{2}{7em}{continuous ${P}_1$ FCT visc.} & $L^2$ & 2.00E-02 & 4.94E-03 & 1.13E-03 & 2.60E-04 & 6.16E-05\\
& Rate & -- & 2.08 & 2.16 & 2.13 &2.18 \\
\bottomrule
\end{tabular}
}
\end{table}

\subsection{Inflow boundary case}
We consider the domain $\Omega = [0,1]^2$ with a uniform velocity field $\vec{\beta}=(1,1)^T$ and the initial condition $u_0(x,y)=\sin(\pi(x+y))$. The exact solution is given by
\begin{equation}u(x,y;t) = \sin(\pi(x+y-2t)).\end{equation}

This test involves an inflow boundary, for which we impose Dirichlet conditions on the entire boundary for simplicity. We test five mesh sizes $h = 0.05, 0.025, 0.0125$, $0.00625$ and $0.003125$, and the results are summarized in \Cref{test4}. The CR element methods with greedy and local FCT viscosities achieve only first-order accuracy, reflecting the challenge of handling inflow boundaries. In contrast, the global FCT viscosity method retains second-order accuracy, thanks to the specially designed treatment described in \Cref{sec:Inflow} for the nonzero term $L$ in (\ref{forward_inflow}). This demonstrates the robustness of the global FCT approach in preserving high-order accuracy even in the presence of inflow boundaries.
\begin{table}[!t]
    \centering
\caption{Inflow boundary case, convergence rates of CR element with greedy, local FCT and global FCT viscosities.}
\label{test4}
    \resizebox{0.65\textwidth}{!}{
\begin{tabular}{ccccccc}
\toprule
     & h & 0.05 & 0.025 &0.0125 &0.00625& 0.003125\\
\midrule
\multirow{2}{7em}{Greedy visc.} & $L^2$ & 7.17E-02 & 3.39E-02 & 1.61E-02 & 7.76E-03 & 3.80E-03\\
& Rate & -- & 1.08 & 1.08 & 1.05 &1.03 \\
Reconstruction & $L^2$ & 6.71E-02 & 3.21E-02 & 1.55E-02 & 7.59E-03 & 3.75E-03\\
& Rate & -- & 1.06 & 1.05 & 1.03 &1.02 \\
\midrule
\multirow{2}{7em}{Loc. FCT visc.} & $L^2$ & 4.06E-02 & 1.66E-02 & 6.92E-03 & 3.10E-03 & 1.44E-03\\
& Rate & -- & 1.29 & 1.26 & 1.16 & 1.11 \\
Reconstruction & $L^2$ & 3.87E-02 & 1.58E-02 & 6.73E-03 & 3.04E-03 & 1.42E-03\\
& Rate & -- & 1.29 & 1.24 & 1.15 &1.09 \\
\midrule
\multirow{2}{7em}{Glob. FCT visc.} & $L^2$ & 2.93E-03 & 7.40E-04 & 1.86E-04 & 4.66E-05 & 1.17E-05\\
& Rate & -- & 1.99 & 1.99 & 2.00 &2.00 \\
Reconstruction & $L^2$ &4.64E-03 & 1.19E-03 & 3.00E-04 & 7.55E-05 & 1.89E-05\\
& Rate & -- & 1.96 & 1.98 & 1.99 &2.00 \\
\bottomrule
\end{tabular}
}
\end{table}

\subsection{Solid body rotation}
Let the domain be $\Omega = [-1,1]^2$,  the velocity field be  $\vec{\beta}=2\pi(-y,x)^T$, the initial condition be
\begin{equation}
	u_0(\mathbf{x}) = 
\begin{cases} 
1 & \text{if $\|\mathbf{x} - \mathbf{x}_d\| \leq r_0$ and $(|x| \geq 0.05$ or $y \geq 0.7$),} \\
1 - \frac{\|\mathbf{x} - \mathbf{x}_c\|}{r_0} & \text{if $\|\mathbf{x} - \mathbf{x}_c\| \leq r_0$}, \\
g(\|\mathbf{x} - \mathbf{x}_h\|) & \text{if $\|\mathbf{x} - \mathbf{x}_h\| \leq r_0$}, \\
0 & \text{otherwise},
\end{cases}
\end{equation}
where $r_0=0.3$, $g(r)=\frac14(1+\cos(\pi\min(\frac{r}{r_0},1)))$,
$\mathbf{x}_d=(0,0.5)$, $\mathbf{x}_c=(0,-0.5)$, $\mathbf{x}_h=(-0.5,0)$.
The exact solution is
\begin{equation}u(t,x,y)=u_0(x\cos(2\pi t)+y\sin(2\pi t),-x\sin(2\pi t)+y\cos(2\pi t)).\end{equation}
The shape described by $u_0$ is composed of three solids: a slotted cylinder with height 1, a smooth hump with height $\frac12$, and a cone with height 1 \cite{MR1388492}.

We test five mesh sizes $h = 0.1, 0.05, 0.025, 0.0125,$ and $0.00625$, with the results summarized in \Cref{test3}. The first three methods have similar behaviour and exhibit convergence rates higher than $0.4$. Again the errors of the first three methods on the $h$-meshes are comparable to those from continuous ${P}_1$ with entropy viscosity \cite{MR3171280} on the $\frac{h}2$-meshes. \Cref{fig:mesh_results} presents the reconstructed solutions using the CR element method with global FCT viscosity on meshes of decreasing size $h=0.05,0.025,0.0125$. As the mesh is refined, the solution shows improved shape preservation and satisfies the maximum principle throughout the domain. \Cref{fig:two_images}  compares the solutions of the continuous $P_1$ and CR methods, both employing  global FCT viscosity on a mesh with $h=0.0125$. The results demonstrate that the CR method more accurately captures the shapes of both the slotted cylinder and the cone.
\begin{table}[htbp]
    \centering
	\caption{Solid body rotation, convergence rates:  CR element (greedy/local FCT/global FCT viscosity) vs. continuous $P_1$ element (entropy viscosity).}
\label{test3}
    \resizebox{0.65\textwidth}{!}{
\begin{tabular}{ccccccc}
\toprule
     & h &0.1& 0.05 & 0.025 & 0.0125 &0.00625\\
\midrule
\multirow{2}{7em}{Greedy visc.} & $L^2$ & 3.46E-01 & 2.32E-01 & 1.45E-01 & 1.02E-01 & 7.73E-02\\
& Rate & -- & 0.57 & 0.68 & 0.50 &0.41 \\
Reconstruction & $L^2$ & 3.46E-01 & 2.35E-01 & 1.47E-01 & 1.04E-01 & 7.81E-02\\
& Rate & -- & 0.56 & 0.68 & 0.50 & 0.41 \\
\midrule
\multirow{2}{7em}{Loc. FCT visc.} & $L^2$ & 3.00E-01 & 2.12E-01 & 1.39E-01 & 1.00E-01 & 7.58E-02\\
& Rate & -- & 0.50 & 0.61 & 0.47 & 0.41 \\
Reconstruction & $L^2$ &3.02E-01 & 2.15E-01 & 1.41E-01 & 1.02E-01 & 7.66E-02\\
& Rate & -- & 0.49 & 0.61 & 0.48 & 0.41 \\
\midrule
\multirow{2}{7em}{Glob. FCT visc.} & $L^2$ & 2.69E-01 & 1.93E-01 & 1.34E-01 & 9.98E-02 & 7.44E-02\\
& Rate & -- & 0.48 & 0.53 & 0.42 & 0.42 \\
Reconstruction & $L^2$ &2.71E-01 & 1.96E-01 & 1.36E-01 & 1.01E-01 & 7.51E-02\\
& Rate & -- & 0.47 & 0.53 & 0.43 & 0.42 \\
\midrule
\multirow{2}{7em}{continuous ${P}_1$ entropy visc.} & $L^2$ & 3.47E-01 & 2.54E-01 & 1.96E-01 & 1.48E-01 & 1.14E-01\\
& Rate & -- & 0.47 & 0.38 & 0.41 & 0.38 \\
\bottomrule
\end{tabular}
}
\end{table}

\begin{figure}[!t]
    \centering
    \begin{minipage}[t]{0.32\textwidth}
        \centering
        \includegraphics[width=\textwidth]{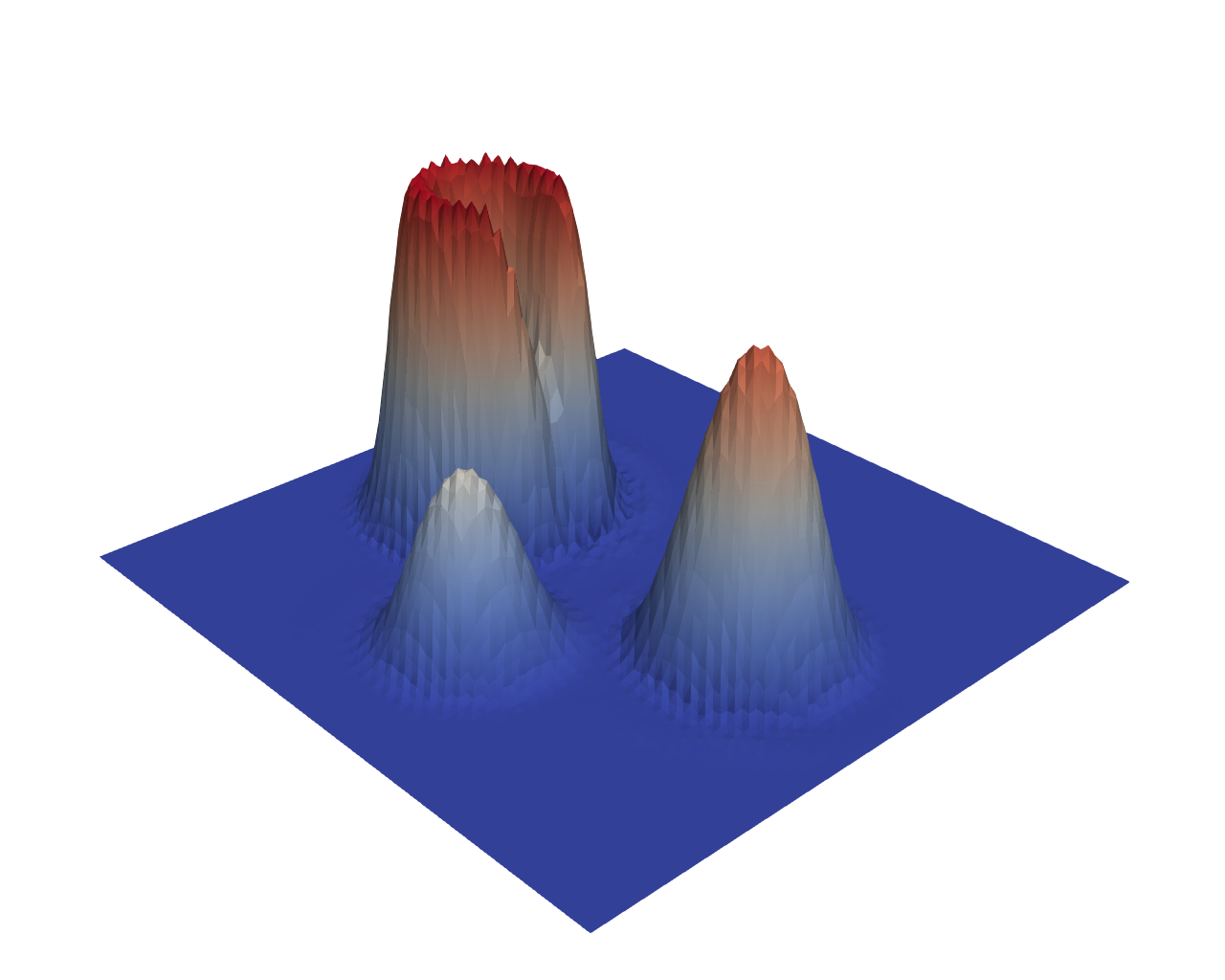}
        \small 
    \end{minipage}
    \hfill
    \begin{minipage}[t]{0.32\textwidth}
        \centering
        \includegraphics[width=\textwidth]{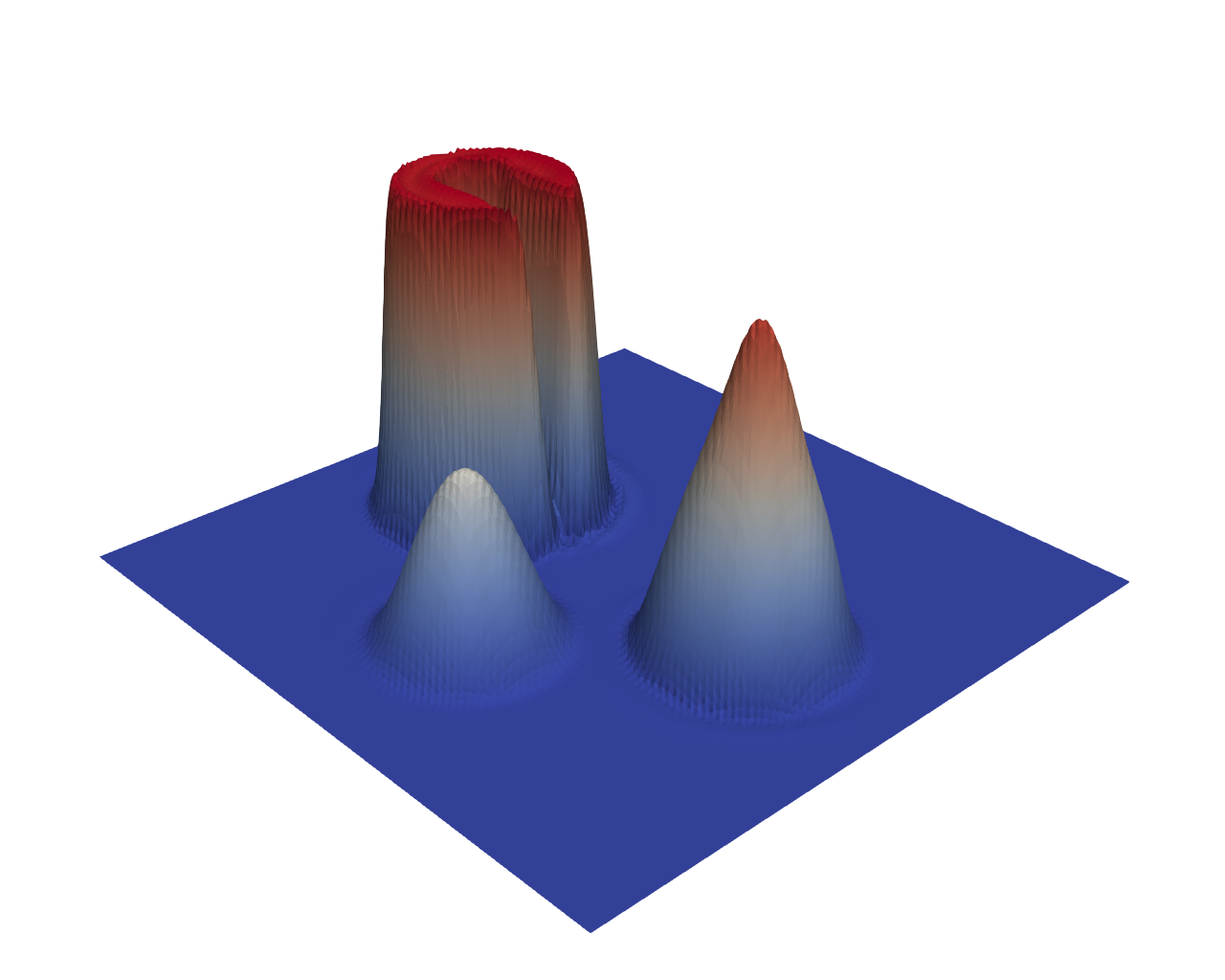}
        \small
    \end{minipage}
    \hfill
    \begin{minipage}[t]{0.32\textwidth}
        \centering
        \includegraphics[width=\textwidth]{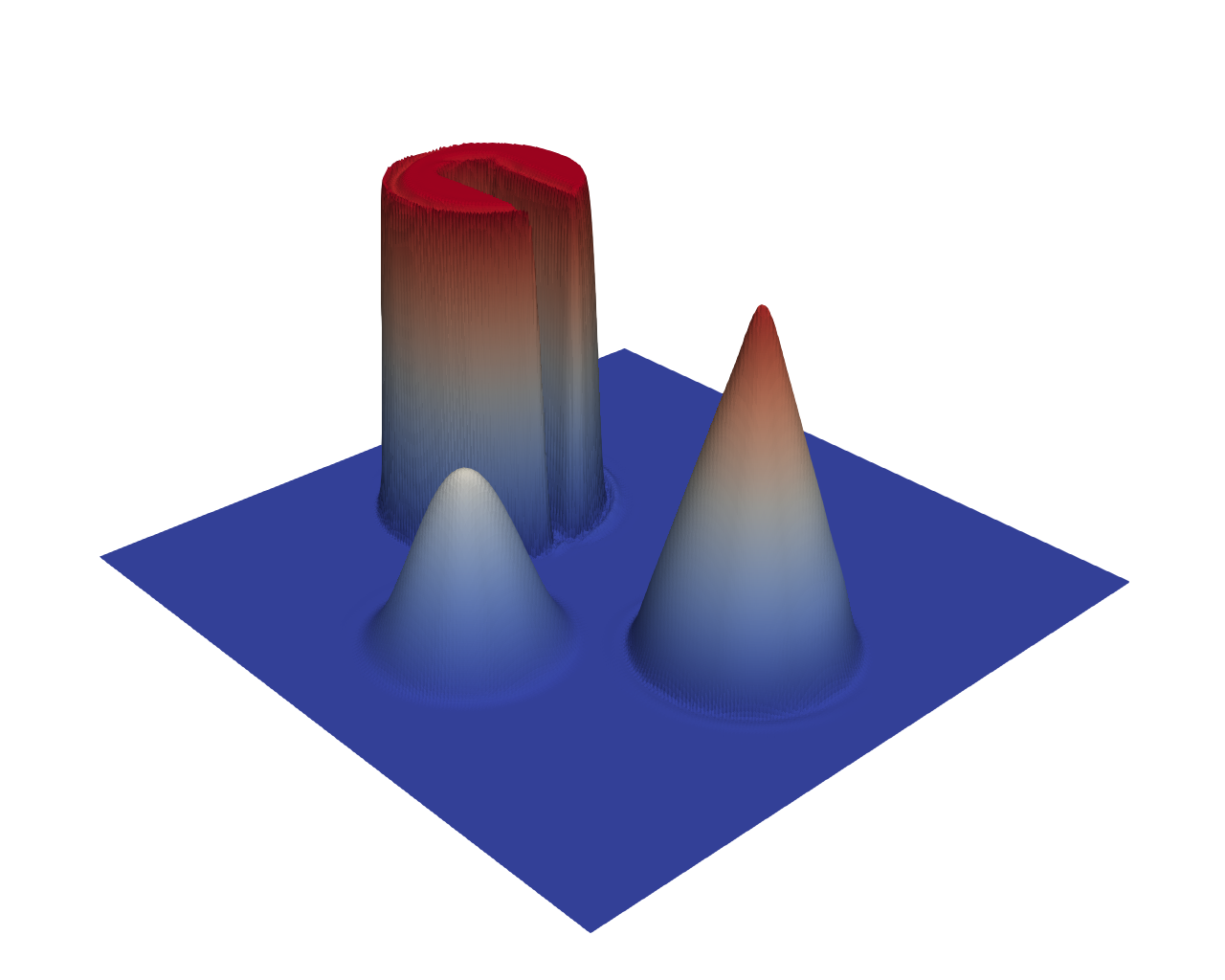}
        \small
    \end{minipage}

    \caption{Solid body rotation at $t=1$: CR method with global FCT viscosity. Mesh sizes $h=0.05,0.025,0.0125$ (left to right). Reconstruction via Wachspress coordinates.}
    \label{fig:mesh_results}
\end{figure}

    

\begin{figure}[!t]
    \centering
    \begin{minipage}[t]{0.49\textwidth}
        \centering
        \includegraphics[width=\textwidth]{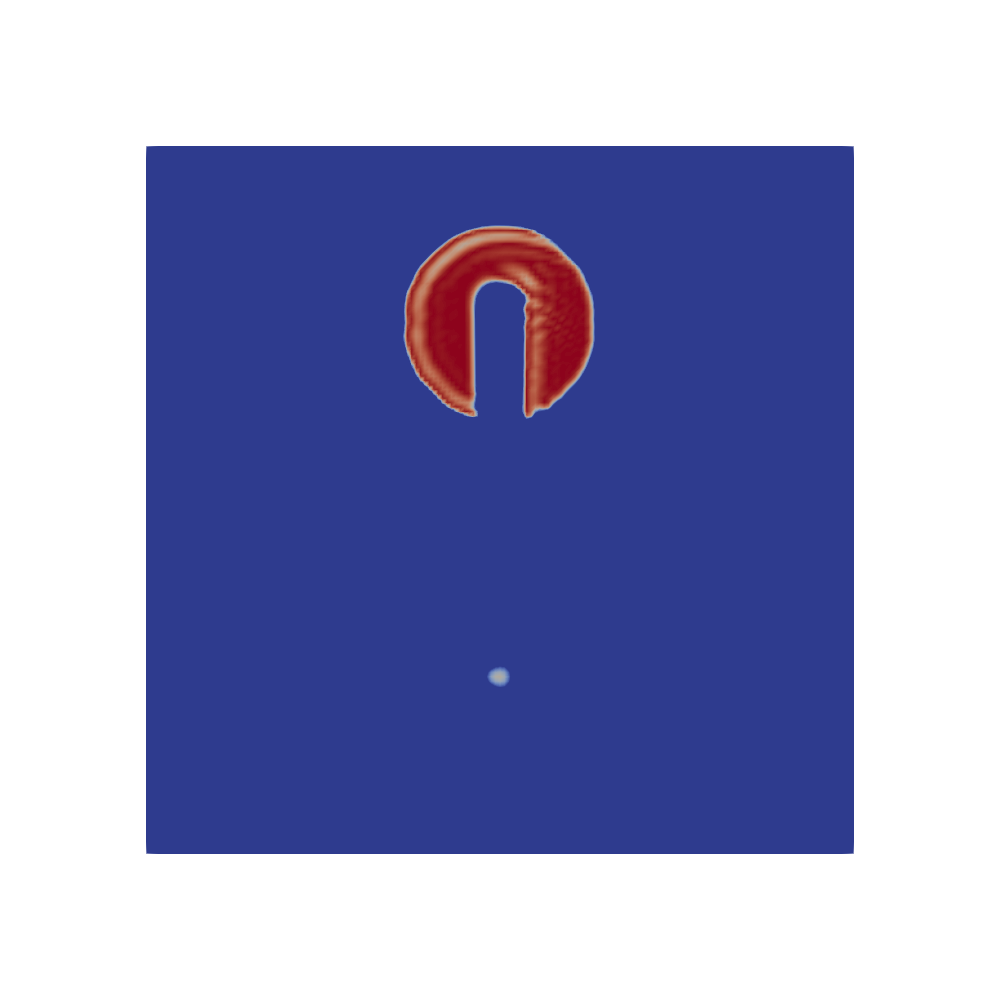}
    \end{minipage}
    \begin{minipage}[t]{0.49\textwidth}
        \centering
        \includegraphics[width=\textwidth]{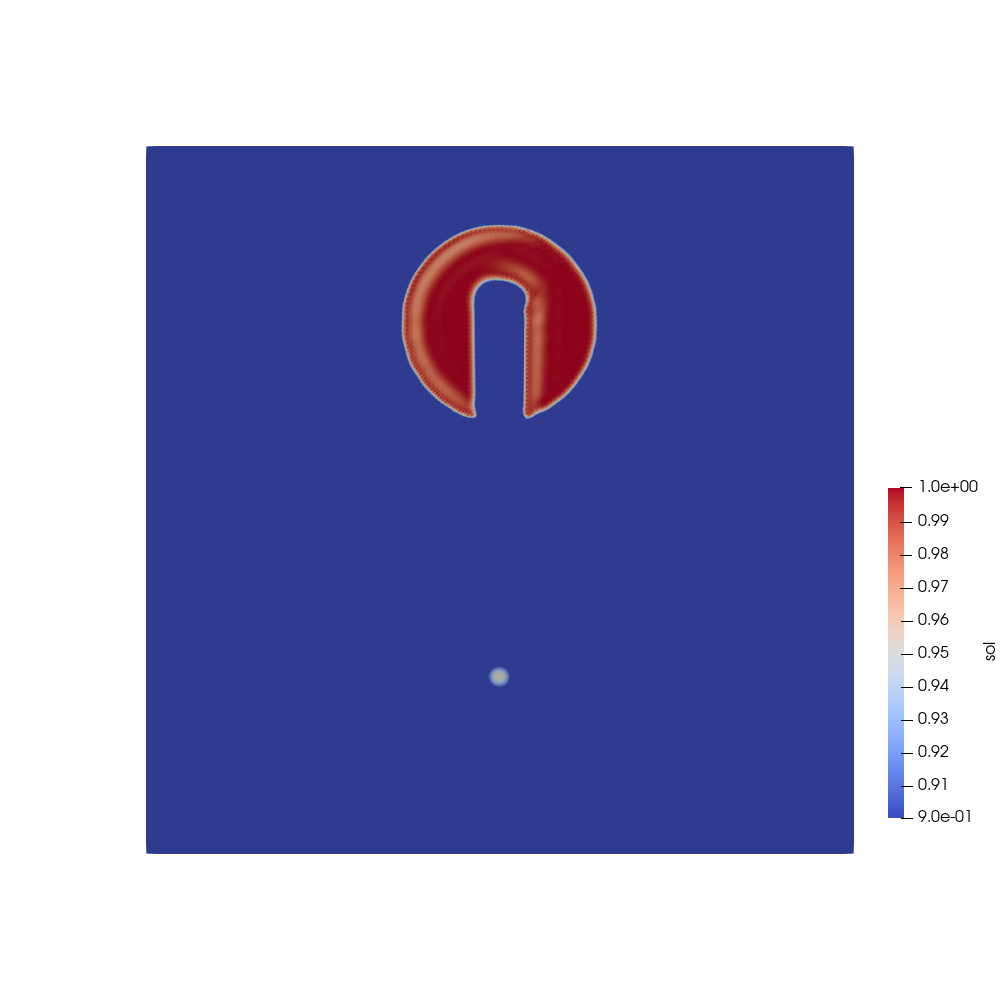}
    \end{minipage}
    
    \caption{Solid body rotation at $t=1$: continuous $P_1$ (left) vs.  CR method (right), both with global FCT viscosity, $h=0.0125$.}
    \label{fig:two_images}
\end{figure}

\subsection{Non-solenoidal compressive rotational case}

Let $\Omega=[0,1]^2$ and now consider a non-solenoidal test with 
velocity field 
\begin{equation}
\boldsymbol{\beta}(x,y)=
\begin{pmatrix}
\sigma(x-c_x)-\omega(y-c_y)\\
\omega(x-c_x)+\sigma(y-c_y)
\end{pmatrix},
\end{equation}
where $\sigma=-0.6$, $\omega=4$, and $(c_x,c_y)=(0.5,0.5)$. Since
\[
\nabla\cdot\boldsymbol{\beta}=2\sigma=-1.2\le 0,
\]
the flow is non-solenoidal and consists of a rotation combined with a uniform compression toward the center of the domain.
The initial condition is
\begin{equation}
u_0(x,y)=
\exp\!\left(-120\bigl((x-0.72)^2+(y-0.50)^2\bigr)\right)\cos(10\pi x).
\end{equation}
This corresponds to a localized oscillatory wave packet. The exact solution is obtained by the method of characteristics:
\begin{equation}
u(x,y,t)=u_0(x_0,y_0),
\end{equation}
where
\begin{equation}
	\begin{cases}
		x_0=0.5+e^{-\sigma t}\Big(\cos(\omega t)(x-c_x)+\sin(\omega t)(y-c_y)\Big),\\
		y_0=0.5+e^{-\sigma t}\Big(-\sin(4t)(x-c_x)+\cos(\omega t)(y-c_y)\Big).
	\end{cases}
\end{equation}

This test is more challenging than standard solenoidal transport benchmarks, as the wave packet experiences progressive compression while being advected along a rotational flow. Its localized Gaussian envelope combined with high-frequency oscillations makes the solution particularly sensitive to numerical dissipation and dispersion. Five mesh sizes, $h = 0.05, 0.025, 0.0125, 0.00625, 0.003125$, are tested at $t=0.5$ with CFL $0.1$ (\Cref{test5}). Both the CR solution and its reconstruction achieve second-order convergence in $L^2$, with consistently smaller errors than the continuous $P_1$ method, while also attaining higher convergence rates in $L^\infty$. Long-time simulations up to $t=2$ on $N=80$ (see \Cref{nonsolenoid}) show that the CR solution remains well-behaved, whereas the continuous $P_1$ method exhibits more pronounced dispersion errors.

    \begin{table}[!t]
        \centering
        \caption{Non-solenoidal compressive rotational test, convergence rates:  CR element vs. continuous $P_1$ element, both with global FCT viscosity.}
        \label{test5}
        \resizebox{0.78\textwidth}{!}{
        \begin{tabular}{cccccccc}
        \toprule
        & $h$ & 0.05 & 0.025 & 0.0125 & 0.00625&0.003125 \\
        \midrule
        \multirow{4}{7em}{CR}
& $L^2$ & 5.30E-02 & 2.27E-02 & 3.74E-03 & 7.10E-04 & 1.38E-04 \\
& Rate & -- & 1.22 & 2.60 & 2.40 & 2.37 \\
& $L^{\infty}$ & 8.38E-01 & 4.70E-01 & 1.44E-01 & 4.52E-02 & 1.49E-02 \\
& Rate & -- & 0.83 & 1.71 & 1.67 & 1.60 \\
        \cmidrule{2-7}
        \multirow{4}{7em}{Reconstruction}
& $L^2$ & 5.33E-02 & 2.36E-02 & 3.78E-03 & 7.12E-04 & 1.41E-04 \\
& Rate & -- & 1.18 & 2.64 & 2.41 & 2.34 \\
& $L^{\infty}$ & 8.58E-01 & 4.70E-01 & 1.44E-01 & 4.52E-02 & 1.49E-02 \\
& Rate & -- & 0.87 & 1.71 & 1.67 & 1.60 \\
        \midrule
        \multirow{4}{7em}{continuous ${P}_1$}
& $L^2$ & 5.65E-02 & 3.35E-02 & 8.45E-03 & 2.00E-03 & 3.94E-04 \\
& Rate & -- & 0.76 & 1.99 & 2.08 & 2.34 \\
& $L^{\infty}$ & 8.55E-01 & 5.83E-01 & 2.23E-01 & 8.53E-02 & 2.84E-02 \\
& Rate & -- & 0.55 & 1.39 & 1.38 & 1.58 \\
        \bottomrule
        \end{tabular}
        }
        \end{table}

\begin{figure}[!t]
    \centering
    \includegraphics[width=\textwidth]{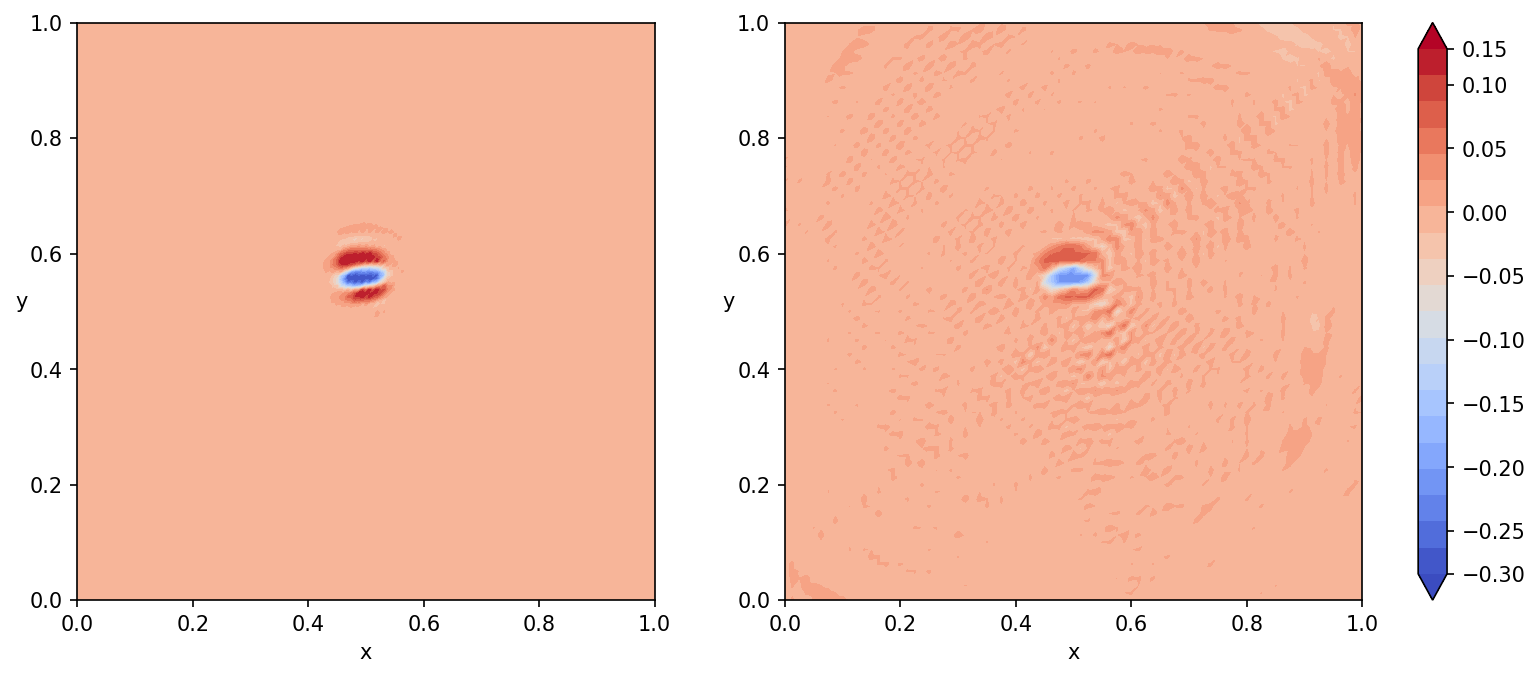}

    \caption{Numerical solutions for the non-solenoidal compressive rotational test at $t=2$. The left and right panels show the CR and continuous $P_1$ solutions, respectively, both with global FCT viscosity on the mesh with $N=80$. The CR method produces a more robust solution with less dispersion compared to the continuous $P_1$ method.}
    \label{nonsolenoid}
\end{figure}
\section{Conclusions}
\label{sec:conclusions}

In this paper, we showed that the Crouzeix-Raviart element provides an effective framework for constructing explicit, second-order, maximum-principle-preserving schemes for time-dependent transport equation. The proposed viscosity mechanisms enforce a local discrete maximum principle at the degree-of-freedom level, while the Wachspress-based reconstruction extends the bound-preserving property to the whole domain. In the divergence-free case, the resulting schemes are also conservative.
The numerical results, including smooth and discontinuous test cases and examples with solenoidal and non-solenoidal velocity fields, demonstrate that the proposed methods are robust and accurate.
The forthcoming sequel will design CR methods for scalar conservation laws, developing appropriate counterparts to the techniques introduced herein.

\bibliographystyle{els-cas-templates/cas-model2-names}
\bibliography{references}
\end{document}

%% file: ex_shared.tex
\usepackage{mathtools}
\usepackage{amsthm}
\usepackage{epstopdf}
\usepackage[nameinlink,noabbrev]{cleveref}
\usepackage{ifpdf}
\ifpdf
  \DeclareGraphicsExtensions{.eps,.pdf,.png,.jpg}
\else
  \DeclareGraphicsExtensions{.eps}
\fi

\numberwithin{equation}{section}

\newtheorem{theorem}{Theorem}[section]
\newtheorem{lemma}[theorem]{Lemma}

\theoremstyle{remark}
\newtheorem{remark}[theorem]{Remark}

\crefname{hypothesis}{Hypothesis}{Hypotheses}

%% file: commands.tex
\newcommand{\p}{\partial}
\newcommand{\R}{\mathbb{R}}

\newcommand{\norm}[1]{\|#1\|}